\documentclass[11pt]{amsart}

\usepackage[style=alphabetic]{biblatex}
\usepackage{amssymb,amsfonts}
\usepackage{amsmath,amscd}
\usepackage[all,arc]{xy}
\usepackage{enumerate}
\usepackage{mathrsfs}
\usepackage[pdftex]{graphicx}
\usepackage[utf8]{inputenc}
\usepackage[T1]{fontenc}
\usepackage[usenames,dvipsnames]{color}
\usepackage[bookmarks=false]{hyperref}
\usepackage{dsfont}
\usepackage{tikz-cd}
\usetikzlibrary{automata,positioning}
\usepackage{fullpage}

\usepackage{color}
\usepackage{xcolor}

\numberwithin{equation}{section}

\theoremstyle{plain}
\newtheorem{thm}{Theorem}[section]
\newtheorem{cor}[thm]{Corollary}
\newtheorem{prop}[thm]{Proposition}
\newtheorem{lem}[thm]{Lemma}

\theoremstyle{definition}
\newtheorem{defn}[thm]{Definition}

\newtheorem{example}[thm]{Example}

\newtheorem{aDD^+m}[thm]{ADD^+endum}

\theoremstyle{remark}
\newtheorem{rmk}[thm]{Remark}

\newcommand{\bbC}{\mathbb{C}} 
\newcommand{\bbD}{\mathbb{D}}

\newcommand{\bbP}{\mathbb{P}}

\newcommand{\bbR}{\mathbb{R}} 

\newcommand{\bbZ}{\mathbb{Z}} 

\newcommand*{\defeq}{\mathrel{\vcenter{\baselineskip0.5ex \lineskiplimit0pt
			\hbox{\scriptsize.}\hbox{\scriptsize.}}}%
	=}

\DeclareMathOperator{\volume}{Vol}
\DeclareMathOperator{\VD}{VD}
\DeclareMathOperator{\HD}{HD}

\DeclareMathOperator{\Jac}{Jac}
\DeclareMathOperator{\DD}{DD}

\DeclareMathOperator{\diam}{diam}
\DeclareMathOperator{\dist}{dist}

\title{A Ma\~n\'e-Manning formula for expanding\\
measures
for endomorphisms of $\mathbb P^k$}
\author{Fabrizio Bianchi}
\address{CNRS, Univ. Lille, UMR 8524 - Laboratoire Paul Painlev\'e, F-59000 Lille, France}
 \email{fabrizio.bianchi$@$univ-lille.fr}
\author{Yan Mary He}
\address{Department of Mathematics\\
	University of Oklahoma\\
	Norman, OK 73019}
\email{he@ou.edu}
\date{\today}

\begin{document}

\maketitle
\begin{abstract}
Let $k \ge 1$ be an integer and $f$ a holomorphic endomorphism of $\mathbb P^k (\mathbb C)$
of algebraic degree $d\geq 2$. 
We introduce a {\it volume dimension} for ergodic $f$-invariant probability measures
with strictly positive Lyapunov exponents. In particular, this class of measures includes all ergodic
measures whose measure-theoretic entropy is
 strictly larger than $(k-1)\log d$, a natural generalization of the class of measures of positive
 measure-theoretic entropy in dimension 1.
The volume dimension is equivalent to the  Hausdorff dimension when $k=1$, but depends on
the dynamics of $f$ to incorporate the possible failure of Koebe's theorem and the non-conformality
of holomorphic endomorphisms for $k\geq 2$.
	
If $\nu$ is an ergodic $f$-invariant probability measure with strictly positive Lyapunov exponents, 
we prove a generalization of the Ma\~né-Manning formula relating the volume dimension, the
measure-theoretic entropy, and the sum of the Lyapunov exponents of $\nu$. 
As a consequence, we give a characterization of the first zero of a natural pressure
function for such expanding measures in terms of their volume dimensions. For hyperbolic maps,
such zero also coincides with the volume dimension of the Julia set, and with the exponent
of a natural (volume-)conformal measure. This generalizes results by Denker-Urba\'{n}ski and 
McMullen in dimension 1
to any dimension $k\geq 1$. 

Our methods mainly rely on a theorem by Berteloot-Dupont-Molino, which gives a precise control on
 the distortion of inverse branches of endomorphisms along generic inverse orbits with
 respect to measures with strictly positive Lyapunov exponents.
\end{abstract}

\section{Introduction}
Let $f : \bbP^1(\bbC) \to \bbP^1(\bbC)$ be a rational map of degree $d \ge 2$  and $\nu$ an ergodic $f$-invariant probability measure whose Lyapunov exponent is strictly positive. Such a measure is necessarily supported on the Julia set $J(f)$ of $f$. There is a well-known relation between the Hausdorff dimension $\HD(\nu)$, the measure-theoretic entropy $h_{\nu}(f)$, and the Lyapunov exponent $\chi_{\nu}(f)$ of $\nu$; namely, we have
\begin{equation}\label{eq_MM}
\HD(\nu) = \frac{h_{\nu}(f)}{\chi_{\nu}(f)}.
\end{equation}
This formula is usually referred to as the {\it Ma\~n\'e-Manning formula}; see  \cite{manning1984dimension,Mane}.
Hofbauer and Raith \cite{HofR} proved a version of \eqref{eq_MM}
 for piecewise monotone maps on the unit interval with bounded variation;
  see also \cite{Ledrappier81}.
The fact that \eqref{eq_MM}
 holds in one-dimensional complex dynamics crucially relies on distortion estimates
 for univalent holomorphic maps coming
from Koebe's theorem; see Section \ref{subsec_strategy}.

For smooth dynamical systems in higher dimensions, related formulas are known to hold
in a number of settings. If $f: M \to M$ is a diffeomorphism of a compact manifold $M$ and
$\nu$ is an ergodic probability measure on $M$ 
which is absolutely continuous with respect to the Lebesgue measure, Pesin \cite{Pesin77} proved that
$$h_{\nu}(f) = \chi^+_{\nu}(f)$$
where $\chi^+_{\nu}(f)$ is the sum of the non-negative Lyapunov exponents of $f$
counted with multiplicity; see also \cite{mane80}. 
When $M$ is a surface, Young \cite{Young} proved that 
\[
\HD(\nu)= \frac{h_\nu(f)}{\chi_+} +\frac{h_\nu(f)}{|\chi_-|} 
\] 
 when $\nu$ is ergodic
  and  $\chi_- < 0 < \chi_+$ are its Lyapunov exponents. This formula has been generalized 
 to the case of diffeomorphisms in any dimension;
 see \cite{LedrappierYoung85II} and \cite{BarreiraPesinSchmeling99}.
Such systems display attracting and repelling directions, and 
one decomposes the problem into
two problems, one for $f$ (along unstable manifolds) and one for $f^{-1}$ (along stable manifolds).
The Ma\~n\'e-Manning formula \eqref{eq_MM}
can be seen as a version of Young's result in (complex)
 dimension 1 where the system is not invertible.
 In this paper, we address
the validity of \eqref{eq_MM} in several complex variables, and more specifically for expanding measures
for (non-invertible) holomorphic endomorphisms of projective spaces in any dimension.

\medskip

Let $k \ge 1$ be an integer and denote $\bbP^k \defeq \bbP^k(\bbC)$.
If $f : \bbP^{k} \to \bbP^{k}$ is a holomorphic endomorphism of algebraic degree $d \ge 2$,
 it is not hard to find examples where (\ref{eq_MM}),
 with $\chi_\nu$ replaced by the sum of the Lyapunov exponents of $\nu$ 
(the natural generalization of the expansion rate along generic orbits),
does not hold. For instance, one can consider product self-maps of $\mathbb C^2$
of the form $(z,w)\mapsto (z^2 +a_1, w^2 + a_2)$,
where $a_i \in \bbC$ are such that the measures of maximal entropy of each component
have different Hausdorff dimensions. 

In \cite{BinderDeMarco}, Binder-DeMarco proposed a conjectural formula for the
Hausdorff dimension of the measure of maximal entropy $\mu$ of an endomorphism of $\mathbb P^k$ as follows:
$$\HD(\mu) = \frac{\log d}{\chi_1} + \dots + \frac{\log d}{\chi_k}.$$
This conjecture has been partially settled 
 \cite{BinderDeMarco,DinhDupont04,Dupont11}, 
and also versions of it have been proposed (and partially proved) for more general invariant measures
\cite{DinhDupont04,Dupont11,Dupont12,DeThelinVigny15,DupontRogue20dimension}.
In this paper, we introduce a  natural  dimension $\VD(\nu)$ for
ergodic $f$-invariant measures $\nu$ with strictly positive Lyapunov
exponents and show that this dimension satisfies a natural generalization of
\eqref{eq_MM}, 
where $\chi_\nu$ is replaced by (two times) the sum of the Lyapunov exponents.

\subsection{Statement of results}
Let $f : \bbP^k \to \bbP^k$ be a holomorphic endomorphism of algebraic degree $d \ge 2$.
The \emph{Julia set} $J(f)$ of $f$ is the support of the unique measure of
maximal entropy of $f$ \cite{Lyubich83, BD01, dinh2010dynamics}.
 Let $\mathcal{M}^+(f)$ 
 (resp. $\mathcal M^+_J (f)$)
 be the set of ergodic invariant probability measures on $\mathbb P^k$ (resp. on $J(f)$)
  with strictly positive Lyapunov exponents. The set $\mathcal{M}_J^+(f)$ contains
  the set $\mathcal M_e^+(f)$ of all
  ergodic probability measures whose 
  measure-theoretic
entropy is strictly larger than $(k-1)\log d$ \cite{DeThelin08,Dupont12},
 which are the natural generalization of the ergodic
 measures with strictly positive entropy in dimension $1$.
 Large classes of examples of measures in $\mathcal M^+_e(f)$
 were constructed and studied in \cite{Dupont12,UZ13,SUZ14,BDeq1,BDeq2}. 

We introduce a {\it volume dimension} for measures $\nu \in \mathcal{M}^+(f)$; 
see Section \ref{subsec_strategy} for an overview and Section \ref{sec_voldim} for precise definitions.
 The volume dimension is dynamical in nature and generalizes the notion of Hausdorff dimension
 in dimension 1 to higher dimensions to incorporate the failure of Koebe's theorem
 and the non-conformality of holomorphic endomorphisms.

For $\nu \in \mathcal M^+ (f)$, we denote by $\VD(\nu)$ the volume dimension,
$h_{\nu}(f)$ the measure-theoretic entropy, and $L_{\nu}(f)$ the {\it sum} of the Lyapunov
exponents of $\nu$. The main result of this paper 
relates these three quantities and generalizes
the Ma\~né-Manning formula to any $k \ge 1$.

\begin{thm}\label{thm_main}
Let $f : \bbP^k \to \bbP^k$ be a holomorphic endomorphism of algebraic degree $d \ge 2$.
For every $\nu\in \mathcal M^+(f)$ we have 
$$\VD(\nu) = \frac{h_{\nu}(f)}{2L_{\nu}(f)}.$$
\end{thm}

When
 $k=1$, 
Theorem \ref{thm_main}
reduces to the Ma\~né-Manning formula \eqref{eq_MM}, as
in this case we have $2\VD(\nu) = \HD(\nu)$; see Proposition \ref{rmk_VD}. The factor $2 = 2k/k$ is due to the fact that we weight
open sets of covers by their volume instead of their diameter and we have $k$ Lyapunov exponents, counting multiplicities.
\medskip

As an application of Theorem \ref{thm_main}, we study a number of natural
dimensions and quantities associated to an endomorphism $f$. 
In dimension 1, these quantities are already defined and well studied; see for example
\cite{DU1,DU2,PU, McMullen}. We first define a {\it dynamical dimension}
$\DD^+_J(f)$
 of $f$ as
\begin{equation*}
\begin{aligned}
	\DD^+_J (f) & \defeq \sup  \left\{ \VD (\nu) \colon \nu \in  \mathcal{M}_J^+ (f) \right\}.
\end{aligned}
\end{equation*}
For $k = 1$, recall that the \emph{pressure function} is defined as
\[
P(t) \defeq \sup_{\nu} \left\{h_\nu(f) - t\chi_{\nu}(f)\right\},
\]
where $t \in \bbR$ and the supremum is taken over the set of invariant probability
measures on $J(f)$. In fact, the supremum can be taken over
$\nu \in \mathcal M^+_J (f) = \mathcal M^+(f)$. This can be seen by combining 
Ruelle's inequality \cite{Ruelle78}
 with a theorem of Przytycki \cite{P93} stating that all invariant measures supported on
 the Julia set of a rational map have non-negative Lyapunov exponent.

For any $k \ge 1$, we define in a similar way a pressure function
 $P_J^+(t)$ as 
\[
P^+_J (t) \defeq \sup\left\{h_\nu(f) - t L_{\nu} (f)\colon \nu \in \mathcal{M}_J^+ (f)\right\}.
\]
By the above, we have $P^+_J(t)=P(t)$ when $k=1$. We remark that, for any $k \geq 2$,
there may exist ergodic probability measures $\nu$ on $J(f)$ with $L_\nu(f)<0$; see
Section \ref{ss:pressure-exp} for examples and further comments. However, 
as in the case of $k=1$,
the pressure function $P_J^+(t)$ is still non-increasing and convex for all $k\geq 1$;
see Lemma \ref{lemma_pressure}. 
 We define
$$p_J^+(f) \defeq \inf \left\{t>0 \colon P_J^+(t)\leq 0\right\}.$$

As a consequence of Theorem \ref{thm_main},
we have the following result which generalizes a theorem
 due to Denker-Urba{\'n}ski \cite{DU1, DU2}
 in the case of rational maps to any dimension. 

\begin{thm} \label{thm_main_equalities_general}
	Let $f\colon\mathbb P^k\to \mathbb P^k$ be a holomorphic endomorphism
	of algebraic degree $d\geq 2$.  Then we have
	\[
	2\DD_J^+ (f)  = p_J^+(f).
	\]
\end{thm}

Finally, in the spirit of the celebrated Bowen-Ruelle formula for hyperbolic maps \cite{Bowen79, Ruelle82}, we give
an interpretation of $p_J^+(f)$, when $f$ is hyperbolic
(i.e., uniformly expanding on $J(f)$; see Section \ref{ss:def_ellipses}) in terms of (volume-)conformal measures.
Given $t \ge 0$,
we say that a probability measure $\nu$ on $J(f)$ is {\it $t$-volume-conformal on $J(f)$}
if, for every Borel subset $A\subset J(f)$ on which $f$ is invertible, we have
\begin{equation*}
	\nu (f (A)) =  \int_A |\Jac f|^t d\nu
\end{equation*}
and define
\[\begin{aligned}
	\delta_J (f)  
	& \defeq \inf \left\{ t \ge 0 \colon \mbox{ there exists  a $t$-volume-conformal
		measure on } J(f)\right\}. 
\end{aligned}\]
For $k=1$, the definitions of $t$-volume-conformal measures and $\delta_J(f)$
reduce to those of conformal measures and conformal dimension for rational maps; see \cite{DU1, DU2, McMullen, PU}.
In this case,
owing to Bowen \cite{Bowen79}, one sees that
\[
\delta_J (f)= p_J^+ (f)= \HD (J(f))
\]
for every hyperbolic rational map $f$ on $\mathbb P^1(\mathbb C)$, and that there exists a unique
ergodic measure $\nu$ on $J(f)$ such that $\HD(\nu)= \HD(J(f))$.
We have here the following result in any dimension,
which further motivates
the definition of the volume dimension as a natural generalization of the Hausdorff dimension for
all 
$k\geq 1$.
Observe that, if $f$ is hyperbolic, every invariant probability measure $\nu$ on $J(f)$ belongs to $\mathcal M^+(f)$.

 \begin{thm} \label{thm_main_equalities_hyp}
 Let $f\colon\mathbb P^k\to \mathbb P^k$ be a hyperbolic holomorphic endomorphism
 of algebraic degree $d\geq 2$.  Then we have
 \[
 \delta_J(f) = p_J^+(f) = 2\VD (J(f)) 	 
 \]
 and there exists a unique ergodic
 measure $\nu$ on $J(f)$ such that $\VD(\nu)= \VD(J(f))$.
 \end{thm}
  
\begin{rmk}\label{rmk_lift}
As all our arguments will be local,
our results apply more generally to the setting of \emph{polynomial-like maps} in any dimension, i.e., proper holomorphic maps 
of the form $f\colon U\to V$, with $U\Subset V\Subset \mathbb C^k$ and $V$ convex \cite{DS03,dinh2010dynamics}.
For a large class of such maps (i.e., those whose topological degree dominates 
all the other dynamical degrees \cite{BDR23}),
an analogue of the inclusion $\mathcal M_e^+(f)\subset \mathcal M^+_J (f)$ in this more general context has been proved in 
\cite{BR22}.

As every endomorphism of $\mathbb P^k$ lifts to a homogeneous polynomial endomorphism
 of $\mathbb C^{k+1}$,
 we can
assume for simplicity that the maps we consider are polynomials. Observe that the Lyapunov exponents of every lifted measure
are the same as those of the original measure, with the addition of an extra exponent $\log d$. Since $\log d > 0$ when $d\geq 2$,
this
does not change the condition
 on the positivity of the Lyapunov exponents.
\end{rmk}

\subsection{Volume dimensions and strategy of the proofs} \label{subsec_strategy}
Let us first recall the idea of the proof of the Ma\~né-Manning formula \eqref{eq_MM} in dimension 1.
It essentially consists of two steps.
\begin{enumerate}
\item  The first step consists of defining a \emph{local dimension} at a point $x$ by setting
\[\delta_x\defeq \lim_{r\to 0}\frac{\log \nu (B (x,r))}{\log r}\]
(whenever the limit exists), where $B(x,r)$ denotes the balls of radius $r$ centred and $x$, and proving that 
the limit is well-defined and equal to
the ratio $h_\nu(f)/\chi_\nu (f)$ for $\nu$-almost every $x$. In particular, $\nu$ is \emph{exact-dimensional}.
\item The second step is to prove that the Hausdorff dimension of $\nu$ must be equal to the common value
of the local dimensions found in the first step \cite{Young}.
\end{enumerate}
Let us describe how the one-dimensional setting plays a crucial role in Step (1). 
By \cite{BK83} and \cite{mane80},
for $\nu$-almost every $x$
we have
\[
h_\nu (f) = \lim_{\kappa \to 0}
 \lim_{n\to \infty} \frac{-\log \nu (B_n (x,\kappa))}{n},\]
where $B_n(x,\kappa)$ is the \emph{Bowen ball
of radius $\kappa$ and depth $n$}. 
This is
defined as 
$$B_n (x,\kappa)\defeq \left\{y\colon |f^{j}(y)-f^j(x)| 
< \kappa, 0\leq j \leq n\right\}.$$
The crucial observation is that, for large $n$,
 the Bowen ball
$B_n (x,\kappa)$ 
is 
comparable (up to precisely
 quantifiable errors) to the ball $B(x, \kappa\, e^{-n\chi_\nu (f)})$ 
 of the same center and radius $\kappa\, e^{-n \chi_\nu (f)}$.
 Fixing a $\kappa_0$ 
 for simplicity, and
 setting $n(r)\sim |\log r|/\chi_\nu (f)$, it then follows that
 \[
  \lim_{r\to 0}\frac{\log \nu (B (x,r))}{\log r}
=
\lim_{r\to 0}
\frac{-\log \nu (B_{n(r)} (x, \kappa_0))}{n(r)}
\frac{n(r)}{-\log r}
= \frac{h_\nu(f)}{\chi_\nu},
 \]
 which in particular shows that $\delta_x$ is well-defined.
The precise relation
between geometric balls and Bowen balls is a consequence of 
 Koebe's theorem and related distortion estimates, which imply
 that
 images of balls by holomorphic maps (and in particular by their inverse branches) are still comparable to balls. 
As a consequence,
in complex 
dimension $1$,
there is a natural interplay between the Hausdorff dimension and the dynamics of a rational map.
Observe in particular that one may define
 the Hausdorff dimension of $\nu$
by using covers consisting of Bowen balls, indexed over their depth $n$, and sending $n$ to infinity; see also \cite{CPZ}.

\medskip

All the above 
 is in
 sharp contrast with the 
 higher-dimensional situation, where,
  due to the lack of conformality of holomophic maps,
  preimages of balls can be arbitrarily distorted, and far from being balls.
  In the best possible scenario (e.g., for hyperbolic product maps),
  the preimages of balls are approximately ellipses whose axes reflect the contraction
  rate of the inverse branches in the different directions. 

On the other hand,
when $\nu \in \mathcal M^+(f)$, a
result by Berteloot-Dupont-Molino (see \cite{BDM, BD} and Theorem \ref{thm_distortion} below)
states that the best possible scenario described above is actually true, in an infinitesimal sense, 
for preimages of balls along generic orbits of $\nu$.
More precisely, there exists an increasing (as $\epsilon\to 0$) measurable exhaustion
$\{Z^\star_\nu (\epsilon)\}_\epsilon$ of a full-measure subset $Z^\star_\nu$ of the space of orbits for $f$ such that
 the preimages of sufficiently small balls along  orbits in $Z^\star_\nu(\epsilon)$ 
 are approximately ellipses, and the contraction rate for their volume
 is essentially given (up to further controllable error terms) by $e^{-nL_\nu(f)+n{O(\epsilon)}}$.
This is 
a consequence of very refined estimates
on the convexity defect of such preimages.
Such property was already exploited in \cite{BB18} to give bounds on the
Hausdorff dimension of the bifurcation locus of families of endomorphisms of $\mathbb P^k$
\cite{BBD18,B19}, and in particular to prove that this is maximal near isolated Lattés maps, i.e.,
maps for which all the Lyapunov exponents are equal and minimal, i.e., equal to $(\log d)/2$ \cite{,BD99,BertelootDupont05}.

Fix $\nu \in \mathcal{M}^+(f)$. Denote by $\pi\colon Z^\star_\nu\to \mathbb P^k$
the projection associating to any orbit $\hat z =\{z_n\}_{n\in \mathbb Z}$
its element $z_0$. For $x \in \pi(Z^\star_\nu(\epsilon))$, $\kappa>0$,
 and $N\in \mathbb N$,
we consider
(when well-defined)
the neighbourhood $U=U(N,x,\kappa,\epsilon)$
 of $x$
  satisfying $$f^N (U) = B(f^N(x),\kappa\, e^{-NM\epsilon})$$ where 
  $e^{M}$ is a bound for the expansion of $f$ and 
  we require that
  $f^N|_{U}$ is injective. 
  It follows from 
   the above 
  result by Berteloot-Dupont-Molino, and by further estimates that we develop in Section \ref{sec_prelim},
  that there exist some $r(\epsilon)$ and $n(\epsilon)$ such that, for all
  $x \in \pi(Z^\star_\nu(\epsilon))$, $0<\kappa<r(\epsilon)$, and $N\geq n(\epsilon)$
the sets   $U(N,x,\kappa,\epsilon)$  are indeed well-defined and
approximately ellipses, of controlled geometry.
  We see these sets $U(N,x,\kappa,\epsilon)$ as a suitable
  version of the Bowen balls $B_n(x,\kappa)$
  in any dimension.
Let us set
\begin{equation*}
	\delta_x (\epsilon,\kappa, N) \defeq \frac{\log \nu(U(N,x,\kappa,\epsilon))}{\log {\rm Vol}(U(N,x,\kappa,\epsilon))},
\end{equation*}
where $\rm Vol$ denotes the volume 
with respect to the Fubini-Study metric.
As a first step 
(which corresponds to Step (1) above)
towards proving Theorem \ref{thm_main}, 
we show that every $\nu\in \mathcal M^+(f)$
 is {\it exact (volume-)dimensional}; namely, for $\nu$-almost every $x$, we have
$$\limsup_{\epsilon \to 0} \limsup_{\kappa \to 0} \limsup_{N \to \infty}\delta_x (\epsilon,\kappa, N) = \liminf_{\epsilon \to 0} \liminf_{\kappa \to 0} \liminf_{N \to \infty}\delta_x (\epsilon,\kappa, N)
=
\frac{h_\nu(f)}{2L_\nu(f)};
$$
 see
 Theorem \ref{thm_11.4.2} 
 and
 Corollary \ref{cor_exactdim}. 
We adapt here the approach of Ma\~n\'e 
\cite{Mane}
in higher dimensions, thanks to the distortion estimates
developed in Section \ref{sec_prelim}.

\medskip

Once the local dimension of every $\nu\in \mathcal M^+(f)$ is well-defined as above, 
we give a global interpretation of this quantity by defining 
a \emph{volume dimension}
for these measures. The idea is to use the sets
$U(N,x,\kappa,\epsilon)$ to cover 
the ``slice'' $X\cap \pi (Z^\star_\nu(\epsilon))$ of every set $X\subseteq Z^\star_\nu$.
More precisely, for every $X \subseteq \pi(Z^\star_\nu)$ and $\epsilon>0$, setting $X^\epsilon \defeq   X\cap \pi(Z^\star_\nu(\epsilon))$, 
 we define the
 quantity
 $\VD^\epsilon_\nu(X^\epsilon)$ as
\[
\VD^\epsilon_\nu(X^\epsilon) \defeq 
\sup \left\{ \alpha :  \Lambda^\epsilon_\alpha(X^\epsilon) = \infty\right\}=
\inf\left\{ \alpha : \Lambda^\epsilon_\alpha(X^\epsilon) = 0\right\},
\]
where
\[
\Lambda^\epsilon_\alpha (X^\epsilon)
\defeq
\lim_{\kappa\to 0}
\lim_{N^\star \to \infty} \inf_{\{U_i\}}
\sum_{i \ge 1} \volume(U_i)^\alpha.
\]
Here 
 the infimum is taken over the covers consisting of sets $U_i$
of the form $U_i=U(N_i, x,\kappa,\epsilon)$, 
  for some
  $x\in \pi(Z_\nu(\epsilon))$ and
  $N_i \ge N^\star$.
The \emph{volume dimensions} of $X$ and $\nu$ are then 
respectively defined as
\[
\VD_\nu(X) \defeq \limsup_{\epsilon \to 0} \VD_{\nu}^{\epsilon}(X^\epsilon)
\quad \mbox{ and } \quad
\VD(\nu)\defeq \inf \left\{\VD_\nu(X)\colon X \subseteq \pi(Z_\nu), \nu(X)=1\right\},
\]
and the $\limsup_{\epsilon\to 0}$ is actually a limit; see Section \ref{subsec_eqdefnl}.
We prove in Proposition \ref{prop_2.1} a version of Young's criterion \cite[Proposition 2.1]{Young},
 relating the local volume dimensions $\delta_x$ with the  volume dimensions $\VD_\nu(X)$ and $\VD(\nu)$.
 This corresponds to 
 Step (2) above and, together with 
 the exact volume-dimensionality of $\nu$  proved in the first step, completes the proof of Theorem \ref{thm_main}.

\subsection{Organization of the paper}
The paper is organized as follows. In Section \ref{sec_prelim}, we derive from the distortion theorem \cite{BDM,BD} the estimates that we will need, and we introduce the 
volume-conformal measures and the pressure function $t \mapsto P_J^+ (t)$. 
We prove the exact dimensionality of every
 $\nu \in \mathcal{M}^+(f)$ in Section \ref{sec_proofmain}.
In Section \ref{sec_voldim}, we define and study the volume dimensions of sets and measures.
We conclude the proof of
Theorem \ref{thm_main}
and prove Theorems \ref{thm_main_equalities_general}
 and \ref{thm_main_equalities_hyp} in Section \ref{sec_proof1.3}.

\subsection*{Acknowledgements}
The authors would like to thank the University of Lille and the University of Oklahoma for the warm welcome and for the excellent work conditions.

This project has received funding from
 the French government through the Programme Investissement d'Avenir
 (LabEx CEMPI ANR-11-LABX-0007-01,
ANR QuaSiDy ANR-21-CE40-0016,
ANR PADAWAN ANR-21-CE40-0012-01)
managed by the Agence Nationale de la Recherche.

\section{Definitions and preliminary results}\label{sec_prelim}
After fixing some notations in Section \ref{ss:def_ellipses}, in Section \ref{subsec_dist}
we recall the
  distortion theorem by Berteloot-Dupont-Molino \cite{BDM,BD} 
  and deduce the estimates which will be essential
  ingredients in the proof of Theorem \ref{thm_main}. We define and study basic properties about volume-conformal measures
  in Section \ref{ss:volume-conformal},
  and a pressure function in Section \ref{ss:pressure-exp}.

\subsection{Notations}\label{ss:def_ellipses}
Let $f : \bbP^k \to \bbP^k$ be a holomorphic endomorphism and $\nu$ an ergodic $f$-invariant probability measure. By Oseledets' theorem \cite{Oseledec}, one can associate to $\nu$ its \emph{Lyapunov exponents} $\chi_{\min} \defeq \chi_l < \ldots <\chi_1$,
 where $1\le l \le k$. For $\nu$-almost every $x \in \bbP^k$, there exists a stratification in
 complex
 linear subspaces
 $\{0\}=: {(L_{l+1})_x} \subset {(L_l)_x} \subset \ldots \subset (L_1)_x = T_x \mathbb P^k$ of the complex
 tangent space $T_x \mathbb P^k$ such that $Df_x (L_j)_x = (L_j)_{f(x)}$ and $\lim_{n\to \infty} n^{-1}\log ||Df_x^n v|| 
 =\chi_j$
 for all $v \in (L_{j})_x \setminus (L_{j+1})_x$ for all $1\leq j \leq l$.
 
Let us first assume that all the $\chi_j$'s are distinct, i.e., that we have $l=k$ and $\chi_{\min} = \chi_k < \ldots < \chi_1$.
Then,
  $(L_j)_x$ has dimension $k-j+1$ for all $1\leq j \leq k$. 
We denote by $\mathcal O$ a full measure subset of the support of $\nu$
 given by Oseledets' theorem.  
Take
$x \in \mathcal O$. Fix  a basis $(\ell_j)_x$ of the complex
tangent space $T_x \mathbb P^k$ with the property that 
 $(L_j)_x$
 is equal to the span of $\{(\ell_j)_x, \dots, (\ell_k)_x \}$.
Denote by $\{e_j\}_{j=1}^k$ the standard basis of $\mathbb R^k \subset \mathbb C^k$. For every $r_1, \dots, r_k \in \mathbb R$ (sufficiently small), we denote by
$\mathcal E_{x} (r_1, \dots, r_k)$
the image of the unit ball $\mathbb B^k \subset \mathbb C^k$ 
(in a given local chart at $x$) under the composition $e\circ \Phi : \bbC^k \to \bbP^k$, where $e\colon T_x \mathbb P^k\to \mathbb P^k$ is the standard exponential map and $\Phi \colon \mathbb C^k\to T_x \mathbb P^k \simeq \mathbb C^k$ is a linear map such that
$\Phi \big((e_j)_x\big) = r_j (\ell_j)_x$.

If, for all
$1\leq j\leq k$, 
the argument $r_j$ of $\mathcal E_x$ is a function $\phi(j)$ depending on $j$, we write
\[
\mathcal E_x (\phi(j)) \defeq \mathcal E_x (\phi(1), \dots, \phi (k))
\]
for brevity.
In particular, we will often
have $\chi_{\min}>0$ and
 take $\phi(j)$ of the form
$\phi(j)= c_1 e^{-n(\chi_j\pm c_2\epsilon)}$ 
for some 
$n\in \mathbb N$, $0<\epsilon\ll\chi_{\min}$
 sufficiently small,
 and some positive constants
 $c_1$
 (independent of $j$ and $n$) and $c_2$
 (independent of $j$, $n$, and $\epsilon$).
We will call the sets $\mathcal E_x$ \emph{dynamical ellipses} in this case.

If $l<k$, i.e.,
some Lyapunov exponent $\chi_j$ has multiplicity larger than 1, the above construction 
 generalizes by taking into account the corresponding $r_j$ with the same multiplicity. Namely, we assign the same 
$r_j$
to all the directions associated to the same Lyapunov exponent which has multiplicity larger than 1.

\medskip

Let $X\subseteq \mathbb P^k$ be a closed invariant set for $f$.
We denote by $\mathcal M_X^+ (f)$ the set of all ergodic $f$-invariant measures supported on $X$
with strictly positive Lyapunov exponents. We drop the index $X$ if $X = \mathbb{P}^k$.
We say that $X$ is {\it uniformly expanding} if there exist $\eta > 1$ and $C>0$ such that
$||Df_x^n(v)|| > C\eta^n||v||$ for every $x \in X$, $v \in T_x \mathbb P^k$, and $n\in\mathbb N$.
We say that $f$ is {\it hyperbolic} if $J(f)$ is uniformly expanding.

We will consider the Fubini-Study metric on $\mathbb P^k$. We will denote by $\dist$ the corresponding distance,
and by $B(x,r)$ the open ball centred at $x$ and of radius $r$.
For an open set $V \subset \bbP^k$, we denote by $\volume(V)$ the volume of $V$
with respect to the Fubini-Study metric. Given a holomorphic map $g\colon V\to \mathbb P^k$,
we denote by $\Jac g(x)$ the Jacobian of $g$ at $x \in V$, i.e., the determinant of the differential $Dg_x$. 

We also fix the positive constant $M>0$ defined as
\begin{equation}\label{eq:def-M}
M\defeq  \log 
 \sup_{x \in \mathbb P^k} \sup_{v \in \mathbb C^k} \frac{||Df_x (v)||}{||v||}
\end{equation}
and observe that $e^{M}$ dominates the Lipschitz constant of $f$. In particular,  we have
$\dist (f(x_1),f(x_2))\leq e^M \dist (x_1,x_2)$ for every $x_1,x_2 \in \mathbb P^k$,
and $f (B(x,r))\subseteq B(f(x), e^{M} r)$ for every $x \in \mathbb P^k$ and $r>0$.

\subsection{Distortion estimates along generic inverse branches} \label{subsec_dist}
We fix in this subsection a holomorphic endomorphism $f : \bbP^k \to \bbP^k$
of algebraic degree $d \ge 2$ and a measure $\nu\in \mathcal M^+ (f)$.
All the objects and the constants that we introduce in this subsection depend on $f$ and $\nu$.
   We denote by $ \chi_1>\ldots > \chi_l = \chi_{\min}>0$
   the (distinct)
   Lyapunov exponents  of $\nu$, by $k_1, \dots, k_l$ their
respective multiplicities,
 and by $L_\nu = L_\nu(f) \defeq \sum_{j=1}^{l}k_j\chi_j$ their sum.
Recall that we have $L_\nu = \int \log|\Jac f|d\nu$
by Birkhoff's ergodic theorem.

\medskip
Consider the \emph{orbit space} of $f$
\[O \defeq \left\{\hat{x}
= \{x_n\}_{n \in \bbZ} \in (\mathbb P^k)^{\mathbb Z} \colon x_{n+1} 
= f(x_n)\quad  \forall n \in \mathbb Z\right\}\]
 and the right shift map $T\colon O \to O$
defined as $T(\hat x) = \{x_{n+1} \}_{n\in \mathbb Z}$ for $\hat x =\{x_n\}_{n\in \mathbb Z}$. 
Given $\eta>0$,  a function $\phi \colon O \to (0,1]$ is 
said to be
{\it $\eta$-slow} if for any $\hat{x} \in O$ we have 
$$e^{-\eta}\phi(\hat{x}) \le \phi (T(\hat{x})) \le e^{\eta}\phi(\hat{x}).$$

\medskip

We now recall the construction of the \emph{lift} $\hat \nu$ of $\nu$ to 
$O$; see \cite[Section 10.4]{cornfeld2012ergodic} and \cite[Section 2.7]{PU}.
For $n\in \mathbb Z$, 
we let $\pi_n\colon  O \to \bbP^k$ be the projection map defined by
$\pi_n(\hat{x}) = x_n$, where $\hat x = \{x_n\}_{n\in \mathbb Z}$. We
write $\pi\defeq \pi_0$ for brevity. Observe that $\pi_n \circ T = f\circ \pi_n$ for all $n\in \mathbb Z$.

Consider the $\sigma$-algebra $\hat {\mathcal B}$ on $O$ generated by the sets of the form
$$A_{n,B}\defeq \pi_n^{-1} (B) = \left\{\hat x \colon x_n \in B\right\}$$
with $n\leq 0$ and $B\subseteq \mathbb P^k$ a Borel set.
For all such sets $A_{n,B}$, set $$\hat \nu(A_{n,B})\defeq \nu(B).$$ Then, by the invariance of $\nu$ and the fact that 
$x_n \in B$ if and only if $x_{n-m}\in f^{-m} (B)$ with $m\geq 0$, we see that
$\hat \nu$ is well-defined on the sets $A_{n,B}$ as it satisfies
$\hat \nu (A_{n,B})= \hat \nu (A_{n-m,B})$
for all $m\geq 0$. Similarly, for all $m\geq 0$ and Borel sets $B_0, \dots, B_{-m}\subseteq \mathbb P^k$, 
we have
\[
\hat \nu
\big(
\{
\hat x \colon x_0 \in B_0, \ldots, x_{-m} \in B_{-m}
\}
\big)
=
\nu \left(f^{-m} (B_0) \cap f^{-m+1} (B_1) \cap \ldots \cap B_{-m}\right).
\]
We can then extend $\hat \nu$ to a probability measure on $\hat {\mathcal B}$, that we still
denote by $\hat \nu$. By construction, $\hat \nu$
is $T$-invariant and satisfies $\pi_* (\hat \nu)=\nu$. As $\nu$ is ergodic, one can prove that 
$\hat \nu$ is also ergodic.

\medskip

Recall that the critical set $C(f)$ of $f$ is the set of points $x \in \bbP^k$ at which the differential 
$Df_x$ is not invertible. As all the Lyapunov exponents of $\nu$ are finite, 
and their sum is equal to $\int \log |Df|d\nu$,
we have in particular $\nu(C(f))=0$.
Set $$Z \defeq \big\{ \hat{x} \in O :  x_n \notin C(f) \quad \forall n \in \bbZ \big\}.$$
Then the set $Z$ is $T$-invariant and satisfies $\hat \nu(Z) = 1$.
For every $\hat x \in Z$, we denote by $f_{\hat x}^{-n}$ the inverse branch of $f^n$
defined in a neighbourhood of $x_0$ and such that
$f^{-n}_{\hat x} (x_0) = x_{-n}$.

\medskip

The following result is stated in \cite[Theorem A]{BD} (see also \cite[Theorem 1.4]{BDM})
in the case where $\nu$ is the measure of maximal entropy of $f$.
The same statement and proof hold for any measure in $\nu\in\mathcal M^+(f)$,
as stated at the end of the Introduction -- and used in later sections -- of the same paper.

\begin{thm}\label{thm_distortion}
For every $0<2\eta< \gamma \ll \chi_{\min}$ and $\hat \nu$-almost every $\hat{x} \in Z$, there exist
	\begin{enumerate}
	\item  an integer $n_{\hat{x}} \ge 1$ and real numbers $h_{\hat{x}} \ge 1$ and $0<r_{\hat{x}}, \rho_{\hat{x}} \le 1$, 
	\item a sequence $\{\varphi_{\hat{x},n}\}_{n\ge0}$ of injective holomorphic maps
	$$\varphi_{\hat{x},n} : B(x_{-n}, r_{\hat{x}}e^{-n(\gamma+2\eta)}) \to \bbD^k(\rho_{\hat{x}}e^{n\eta})$$
	sending $x_{-n}$ to $0$ and satisfying
	\begin{equation*}
	e^{n(\gamma-2\eta)}
	\dist (u,v) \le |\varphi_{\hat{x},n} (u)- \varphi_{\hat{x},n} (v)| \le e^{n(\gamma+3\eta)}h_{\hat{x}} \dist (u,v)
	\end{equation*}
for every $n\in \mathbb N$ and $u,v\in B(x_{-n}, r_{\hat{x}}e^{-n(\gamma+2\eta)})$;
\item
 a sequence $\{\mathcal{L}_{\hat{x},n}\}_{n \ge 0}$ of linear maps from $\mathbb C^k$
 to $\mathbb C^k$
	which stabilize each $$H_j \defeq \{0\} \times \ldots \times \bbC^{k_j} \times \ldots  \times \{0\},$$ 
	satisfy
	\begin{equation*}
	e^{-n\chi_j+n(\gamma-\eta)}|v| \le |\mathcal{L}_{\hat{x},n}(v)| \le e^{-n\chi_j+n(\gamma+\eta)}|v|
	\quad 
	\mbox{for all } n\in \mathbb N \mbox{ and }
	v \in H_j,
	\end{equation*}
	and such that 
	the diagram 
	\begin{center}
		\begin{tikzcd}
			B(x_0, r_{\hat{x}}) \arrow[r, "f_{\hat{x}}^{-n}"] \arrow[d, "\varphi_{\hat{x},0}"]
			& B(x_{-n}, r_{\hat{x}}e^{-n(\gamma+2\eta)}) \arrow[d, "\varphi_{\hat{x},n}"] \\
			\bbD^k(\rho_{\hat{x}}) \arrow[r, "\mathcal{L}_{\hat{x},n}"]
			& \bbD^k(\rho_{\hat{x}}e^{n\eta})
		\end{tikzcd}
	\end{center}
  commutes for all $n \ge n_{\hat{x}}$.
	\end{enumerate}
Moreover, the functions 
$\hat{x} \mapsto h_{\hat{x}}^{-1}, r_{\hat{x}}, \rho_{\hat{x}}$ are measurable and $\eta$-slow on $Z$.
\end{thm}

In particular, 
for every $n\in \mathbb N$ and $\hat x$ as in the statement,
the inverse branch
 $f_{\hat x}^{-n}$ is well-defined on the ball $B(x_0, r_{\hat x})$.

\begin{cor}\label{cor_thm_distortion}
	With the same assumptions and notations as in Theorem \ref{thm_distortion} and Section \ref{ss:def_ellipses},
	for $\hat \nu$-almost all $\hat x\in Z$ and all
	 $t, t_1, \dots, t_k\in (0,1]$,
	  $n \ge n_{\hat{x}}$,
	 and
  $y,w \in B(x_0, r_{\hat{x}})$,
	we have
	\begin{enumerate}
	\item
$e^{-n (L_\nu+10k\eta) }
\leq | \Jac f^{-n}_{\hat x}(y)|\leq e^{-n(L_\nu-10k\eta)} $;
	\item
$e^{-20kn\eta}\leq | \Jac f^{-n}_{\hat x}(y)| \cdot |\Jac f^{-n}_{\hat x}(w)|^{-1} \leq e^{20kn\eta}$;
\item $ \mathcal E_{x_{-n}}( t_j r_{\hat{x}}h^{-1}_{\hat{x}}e^{-n(\chi_{j}+10\eta)})
\subset
f^{-n}_{\hat{x}}(\mathcal E_{x_0}(t_jr_{\hat{x}})) 
\subset 
\mathcal E_{x_{-n}}( t_j r_{\hat{x}}h_{\hat{x}}e^{-n(\chi_{j}-10\eta)})$;
\item $ (tr_{\hat x} 
)^{2k} e^{-2 n (L_{\nu}+10k\eta)}
\leq \volume ( \mathcal E_{x_{-n}}  (t r_{\hat x} e^{-n \chi_j} 
)   )
\leq (t r_{\hat x} 
)^{2k} e^{-2 n (L_{\nu}-10k\eta)}$;

\item
$(t r_{\hat x} h_{\hat x}^{-1} )^{2k} e^{-2n (L_{\nu}+20k\eta)}
 \leq \volume (f^{-n}_{\hat{x}}(B({x_0}, t r_{\hat{x}})) )
\leq (t r_{\hat x}  h_{\hat x})^{2k} e^{-2n( L_{\nu}-20k\eta)}$.
	\end{enumerate}
\end{cor}

\begin{proof}
The assertions 
(1) and (3) follow directly from Theorem \ref{thm_distortion} (2) and (3). The assertion
 (2) follows from (1). The assertion 
(4) follows from the fact that the distances in $ \mathbb P (T_{x_{-n}} \mathbb P^k) \simeq \mathbb P^{k-1}$
between the directions associated to distinct Lyapunov exponents 
at $x_{-n}$ are larger (up to a multiplicative constant independent of $n$) than $e^{-5n\eta}$,
 again by Theorem
\ref{thm_distortion} (2). This allows one
 to compare the volume of $\mathcal E_{x_{-n}}  (t r_{\hat x} e^{-n \chi_j})$
 with that of an ellipse in $\mathbb C^k$, whose axes are parallel
to the coordinate planes. The assertion (5) is a consequence of (3),
 applied with $t_j = t$ for all $j=1,\ldots,k$, and (4).
\end{proof}

\begin{defn}\label{d:Z}
We define $Z_\nu\subseteq Z$ to be the full $\hat \nu$-measure
set of elements $\hat x\in Z$
 satisfying the conditions in Theorem \ref{thm_distortion} and Corollary \ref{cor_thm_distortion}.
For $\alpha>0$, we also define
\begin{equation*}
Z_{\nu,\alpha} \defeq \big\{ \hat x \in Z_\nu \colon n_{\hat x} < \alpha^{-1}, r_{\hat x} > \alpha, h_{\hat x} < \alpha^{-1}  \big\}.
\end{equation*}
\end{defn}

It follows from the 
definition
that, as $\alpha \to 0$, the sets $Z_{\nu,\alpha}$ increase to $Z_\nu$. 
In particular, we have $\hat \nu (Z_{\nu,\alpha})\to 1$ as $\alpha \to 0$.

\begin{cor} \label{cor_cordistortion}
	For every $0<\epsilon\ll \chi_{\min}$ sufficiently small,
	there exist
	 $Z'_\nu(\epsilon) \subseteq Z_\nu$,
	 $n'(\epsilon) \ge 1$, and
	$r(\epsilon) \in (0,1)$
	such that
	\begin{enumerate}
	\item  $\hat \nu(Z(\epsilon)) > 1-\epsilon$;
\item 	$n_{\hat{x}} \leq n(\epsilon)$ and $r_{\hat{x}} \ge r(\epsilon)$
	for all $\hat{x} \in Z'_\nu(\epsilon)$;
	\item
for all 
$t, t_1, \ldots, t_k\in (0,1]$,
$n\geq n(\epsilon)$,
 $\hat x \in Z'_\nu (\epsilon)$,
 and
$y,w \in B(x_0, r(\epsilon))$
we have
\begin{enumerate}
\item
$e^{-n (L_\nu+k\epsilon) }
	\leq | \Jac  f^{-n}_{\hat x}(y)|\leq e^{-n(L_\nu-k\epsilon)}$;
	\item
$e^{-kn\epsilon}\leq | \Jac f^{-n}_{\hat x}(y)| \cdot |\Jac f^{-n}_{\hat x}(w)|^{-1} \leq e^{kn\epsilon}$;
\item $ \mathcal E_{x_{-n}}( t_j r(\epsilon)
e^{-n(\chi_{j}+\epsilon)})
	\subset
	f^{-n}_{\hat{x}}(\mathcal{E}_{x_0}(t_jr(\epsilon))) 
	\subset 
	\mathcal E_{x_{-n}}( t_j r(\epsilon)
	e^{-n(\chi_{j}-\epsilon)})$;
	\item $ ( t r(\epsilon) 
	)^{2k} e^{-2 n( L_{\nu}+ k \epsilon)} 
	\leq \volume ( \mathcal E_{x_{-n}}  (t r(\epsilon) e^{-n \chi_j}
	)   )
	\leq (t r(\epsilon) 
	)^{2k} e^{-2n (L_{\nu}-k \epsilon)}$;
\item
$(t r(\epsilon)  
 )^{2k} e^{-2n (L_{\nu} +k \epsilon)} 
\leq \volume (f^{-n}_{\hat{x}}(B({x_0}, t r(\epsilon))) )\leq
 (t r(\epsilon)  
 )^{2k} e^{-2n (L_{\nu}-k\epsilon)}$,
\end{enumerate}
	\end{enumerate}
	where $n_{\hat x}$, 
	$h_{\hat x}$, and $r_{\hat x}$ are as in Theorem \ref{thm_distortion}.
\end{cor}

\begin{proof}
By choosing $\alpha = \alpha (\epsilon)$ sufficiently small, 
Corollary \ref{cor_thm_distortion} and the 
 Definition 
 \ref{d:Z}
  of $Z_{\nu, \alpha}$
give 
the existence of a set $Z''_\nu (\epsilon) \defeq Z_{\nu, \alpha (\epsilon)}$
and numbers $r(\epsilon)$, $n'(\epsilon)$
 satisfying the 
 properties in the statement, with (3c) and (3e) replaced by 
 \[
\mathcal E_{x_{-n}}( t_j r(\epsilon)
\alpha(\epsilon)
e^{-n(\chi_{j}+\epsilon/2)})
	\subset
	f^{-n}_{\hat{x}}(\mathcal{E}_{x_0}(t_jr(\epsilon))) 
	\subset 
	\mathcal E_{x_{-n}}( t_j r(\epsilon)
\alpha(\epsilon)^{-1}	
	e^{-n(\chi_{j}-\epsilon/2)})
	 \]
 and
 \[
(t r(\epsilon)  
\alpha(\epsilon))^{2k} e^{-2n (L_{\nu} +k \epsilon/2)} 
\leq \volume (f^{-n}_{\hat{x}}(B({x_0}, t r(\epsilon))) )\leq
 (t r(\epsilon) 
 \alpha(\epsilon)^{-1})^{2k} e^{-2n (L_{\nu}-k\epsilon/2)},
 \]
 respectively.
 Since all the Lyapunov exponents  of $\nu$ are strictly positive, the assertion follows up to
 increasing $n'(\epsilon)$.   
\end{proof}

\begin{lem} \label{cor_cordistortion4}
 For every $0<\epsilon\ll \chi_{\min}$ sufficiently small,
 there exist $n(\epsilon)\in \mathbb N$,
 a subset 
$Z_\nu(\epsilon) \subseteq Z'_\nu(\epsilon)$
with $\hat \nu (Z_\nu' (\epsilon)\setminus Z_\nu (\epsilon))<\epsilon$, and, for all $\hat x \in Z_\nu(\epsilon)$,
 a sequence $\{n_l\}_{l \ge 0}= \{n_l (\hat x)\}_{l \ge 0}$ such that
 \begin{enumerate}
 \item $n'(\epsilon) \leq n_0 \leq n(\epsilon)$;
 \item  $n_{l+1}-n_l <\epsilon n_l$ for all $l\geq 0$;
 \item  $T^{n_l}(\hat x)  \in Z'_\nu(\epsilon)$ for all $l\geq 0$,
 \end{enumerate}
 where $Z'_\nu(\epsilon)$ and $n'(\epsilon)$ are as in Corollary \ref{cor_cordistortion}. 
\end{lem}

A version of Lemma \ref{cor_cordistortion4} is essentially proved in \cite[Section 11.4]{PU} in the case of $k=1$.
We will need here
to further get a uniform upper bound for the element $n_0$ associated to any $\hat x\in Z_\nu(\epsilon)$.

\begin{proof}
We first show the existence of a set $Z'''_\nu(\epsilon)\subseteq Z'_\nu(\epsilon)$ with
$\hat \nu ( Z'_\nu(\epsilon)\setminus Z'''_\nu(\epsilon))=0$ and, for any $\hat x \in Z'''_\nu(\epsilon)$,
of a sequence $\{n_l\}_{l\geq 0}=\{n_l (\hat x)\}_{l\geq 0}$
satisfying (2), (3), and $n_0\geq n'(\epsilon)$.

\medskip

For every $n\in \mathbb N$ and $\hat x\in Z$, set
$S_n(\hat x)\defeq \sum_{j=1}^{n-1} {\bf 1}_{Z'_\nu(\epsilon)} \circ T^j(\hat{x})$,
where ${\bf 1}_V$ denotes the characteristic function of $V\subset \mathbb P^k$.
Since $\hat \nu$ is ergodic, by Birkhoff's ergodic theorem
 there exists a measurable set $Z'''_\nu(\epsilon) \subseteq Z'_\nu(\epsilon)$ 
such that $\hat \nu(Z'''_\nu(\epsilon)) = \hat \nu(Z'_\nu(\epsilon))$
and 
\begin{equation}\label{eq_1}
	\lim_{n \to \infty} 
	n^{-1} S_n(\hat x) 
	= \hat \nu(Z'_\nu(\epsilon))
	\quad 
	\mbox{ for every } \hat{x} \in Z'''_\nu(\epsilon).
\end{equation}

Take $\hat x \in Z'''_\nu(\epsilon)$.
By
 \eqref{eq_1}
 and the fact that $\hat \nu (Z'_\nu (\epsilon))>1-\epsilon$,
there exists 
$n^\star = n^\star (\hat x) > 10/\epsilon$ 
such that 
$|n^{-1} S_n(\hat x)- \hat \nu (Z'_\nu (\epsilon))|\leq \epsilon/20$ for all $n\geq n^\star$.
We define $\{n_l\}_{l\geq 0}=\{n_l (\hat x)\}_{l\geq 0}$ to be
  the sequence of integers $n\geq \max \{n^\star, n'(\epsilon)\}$ 
such that 
$T^n(\hat x)\in Z'_\nu (\epsilon)$.

It follows from the definitions
of $S_n(\hat x)$ and of the sequence $\{n_l\}_{l\geq 0}$ 
that $S_{n_{l+1}}(\hat x)= S_{n_l} (\hat x)+1$ for all $l\geq 0$.
Moreover, we also have
\[
 S_{n_l} (\hat x)\leq n_l \big( \hat \nu (Z'_\nu (\epsilon))+ \epsilon/20\big)
\quad \mbox{ and } 
\quad
S_{n_{l+1}}(\hat x ) \geq n_{l+1} \big( \hat \nu (Z'_\nu (\epsilon))- \epsilon/20\big)
\quad
\mbox{ for all } l\geq 0.
\]
We deduce from these inequalities
that, again for all $l\geq 0$,
\[
 \frac{n_{l+1}}{n_l} \leq  
\frac{S_{n_{l+1}}(\hat x ) }{n_l \big( \hat \nu (Z'_\nu (\epsilon))- \epsilon/20\big)}\\
= 
\frac{S_{n_{l}}(\hat x ) +1}{n_l \big( \hat \nu (Z'_\nu (\epsilon))- \epsilon/20\big)}
\leq
\frac{\hat \nu (Z'_\nu (\epsilon))+ \epsilon/20 + 1/n_l 
}{\hat \nu (Z'_\nu (\epsilon))- \epsilon/20}.
\]
Hence, since
$0< \epsilon \ll \chi_{\min}$ and 
 $n_l \geq n^\star > 10/ \epsilon$ for all $l\geq 0$,
 we have
\[
\frac{n_{l+1}-n_l}{n_l} 
\leq \frac{\hat \nu (Z'_\nu (\epsilon))+ \epsilon/20 + 1/n_l 
}{\hat \nu (Z'_\nu (\epsilon))- \epsilon/20} -1
\leq
\frac{\epsilon/10 + \epsilon/10 
}{1-\epsilon - \epsilon/20}<\epsilon.
\]
This gives the existence of a set $Z'''_\nu (\epsilon)$ with the properties stated at the beginning of the proof.

\medskip

For every $N>n'(\epsilon)$, set $Z^N_\nu (\epsilon) \defeq\{\hat x \in Z'''_\nu (\epsilon) \colon n_0(\hat x)\leq N\}$.
The sequence of sets $Z^N_\nu (\epsilon)$ is non-decreasing as $N\to \infty$, and satisfies  
$\cup_{N} Z^N_\nu (\epsilon) = Z'''_\nu (\epsilon)$. Fix $m^\star=m^\star (\epsilon)$
 such that $\hat \nu ( Z'''_\nu (\epsilon) \setminus Z^{m^\star}_\nu (\epsilon))< \epsilon$.
 The assertion follows setting $Z_\nu (\epsilon)\defeq Z^{m^\star}_\nu (\epsilon)$ and $n(\epsilon)\defeq m^\star$. 
\end{proof}

Recall that  $\pi \colon  O \to \bbP^k$ denotes the projection map defined by $\pi(\hat{x}) = x_0$,
 where $\hat x = \{x_n\}_{n\in \mathbb Z}$.

\begin{rmk}\label{r:nl-z}
Observe that the sequence $\{n_l\}_{l\geq 0}= \{n_l (\hat x)\}_{l\geq 0}$
as in Lemma \ref{cor_cordistortion4} only depends on $x_0=\pi(\hat x)$.
In particular, the sequence
$\{n_l\}_{l\geq 0}$ as in Lemma \ref{cor_cordistortion4} is well-defined for every $x \in \pi (Z_\nu(\epsilon))$.
\end{rmk}

\begin{defn}\label{d:U}
Given $N \ge 0$, $x \in \mathbb P^k$, $\kappa>0$, and $\epsilon>0$,
we denote by $U=U(N,x,\kappa, \epsilon)$ the (necessarily unique) set $U$,
if it exists, satisfying $f^N (U) = B(f^N(x),\kappa \, e^{-NM\epsilon})$
(where $M$ is as in \eqref{eq:def-M}) and such that $f^N|_{U}$ is injective.
We call $z(U)$ the \emph{center} of $U$.
\end{defn}

By Definition \ref{d:U}, we have
$U(N,x_0,\kappa, \epsilon)\defeq f_{T^N (\hat z)}^{-N} (B(f^{N}(x_0), \kappa \, e^{-NM\epsilon}))$
for every $\hat x$ with $\pi(\hat x)=x_0$ whenever the inverse branch $f_{T^N (\hat z)}^{-N}$
is well-defined on $B(f^{N}(x_0), \kappa \, e^{-NM\epsilon})$.
Lemma \ref{cor_cordistortion4} 
says that, for all $\epsilon>0$ sufficiently small,
 this happens for every $x_0 \in \pi(Z_\nu (\epsilon))$,
  $0<\kappa<r(\epsilon)$, and $l\geq 0$ if  we take $N=n_l$, where the sequence $\{n_l\}_{l\geq 0}$
 is given by that statement (and depends on $x_0$; see Remark \ref{r:nl-z}).
  In particular, $U(n_l,x_0,\kappa, \epsilon)$ is well-defined under
  such conditions for all $l\geq 0$, and Corollary \ref{cor_cordistortion} can be applied with 
  any $\hat y$ such that $y_0 =f^{n_l} (x_0)$ and $y_{-n_l}=x_0$
  (observe that the factor $e^{-NM\epsilon}$ is not necessary to get this).
We now aim at getting similar estimates valid for all $N\geq n(\epsilon)$.
The factor $e^{-NM\epsilon}$ will need to be introduced for this reason.

\medskip

In the next lemma and in the rest of the paper,
we will only consider $0<\epsilon\ll\chi_{\min}$
sufficiently small as above, and
$Z'_\nu(\epsilon)$, $Z_\nu(\epsilon)$, $r(\epsilon)$, and $n(\epsilon)$
will be as in Corollary \ref{cor_cordistortion} and Lemma \ref{cor_cordistortion4}.
Similarly, for every $\hat z \in Z_\nu (\epsilon)$ (and $z \in \pi(Z_\nu(\epsilon)))$, 
the sequence $\{n_l\}_{l\geq 0}$
is given by Lemma \ref{cor_cordistortion4} (and Remark \ref{r:nl-z}).
Recall that all these definitions depend on $f$ and 
$\nu \in \mathcal M^+ (f)$.
The sets $U(N,x,\kappa,\epsilon)$
are as in Definition \ref{d:U}. 

\begin{lem}\label{l:nu-U-compare}
Fix $0<\epsilon\ll \chi_{\min}$. Then, for all $z\in\pi(Z_\nu(\epsilon))$, $N \geq n(\epsilon)$,
and  $0<\kappa < r(\epsilon)$, the set $U(N,x,\kappa,\epsilon)$ is well-defined and we have
\begin{enumerate}
		\item
		$\mathcal E_{z} (\kappa \,
		 e^{-N(\chi_j+(2M+1)\epsilon )})
		\subseteq
		U(N,x,\kappa,\epsilon)
		\subseteq 
		\mathcal E_{z} (\kappa\,
		 e^{-N (\chi_j- \epsilon )})$;
		\item
		$\kappa^{2k} 
		e^{- 2N
		(L_\nu + k(2M+1)\epsilon)}
		\leq 
		\volume (U(N, x, \kappa, \epsilon)) \leq 
		\kappa^{2k}
		e^{-2N
		( L_{\nu}-k\epsilon)}$;
		\item
		$e^{-N(2M+1)k\epsilon}\leq | \Jac f^N (y)| \cdot |\Jac f^N (w)|^{-1} \leq e^{N(2M+1) k\epsilon}$
		for every $y,w\in U(N,x,\kappa,\epsilon)$,
	\end{enumerate}
	where
	$M$ is as in \eqref{eq:def-M}.
\end{lem}

\begin{proof}
	Fix $x \in \pi (Z_\nu(\epsilon))$
	and $0<\kappa< r(\epsilon)$.
	Then $x$ 
	corresponds to (at least) an orbit 
	$\hat x \in Z_\nu(\epsilon)$ with $\pi(\hat x)=x$. 
	By Lemma \ref{cor_cordistortion4} and Remark \ref{r:nl-z}, 
	there exists a sequence $\{n_l\}_{l \ge 0}$ with $n_0 \leq n(\epsilon)$ and such that
	$n_{l+1}<(1+\epsilon) n_l$ and
	$f^{n_l}(x)  \in \pi(Z'_\nu(\epsilon))$ for all $l\geq 0$.
	By Corollary \ref{cor_cordistortion} (3b), (3c), and (3e),
	for all $l\geq 0$ and  
	$0<t\leq 1$ 
	the set
	\[\widetilde U (n_l, x, t \kappa)
	\defeq 
	f_{T^{n_l} (\hat x)}^{-n_l} \big( B(f^{n_l}(x_0), t\kappa)\big)\]
	is well-defined and we have
	\begin{equation}\label{eq:nl-1}
		\mathcal E_{x} (t
 \kappa\,
		e^{-n_l (\chi_j+ \epsilon)})
		\subseteq
		\widetilde U(n_l,x,t \kappa)
		\subseteq 
		\mathcal E_{x} (t 
		\kappa\,
		 e^{-n_l(\chi_j-  \epsilon)}),
		\end{equation}
	\begin{equation}\label{eq:nl-2}
		(t\kappa)^{2k}
		e^{-2 n_l ( L_{\nu}+k\epsilon)} 
		\leq 
		\volume (\widetilde U(n_l, x, t\kappa)) \leq 
		(t\kappa)^{2k} 
		e^{-2n_l (L_{\nu}- k\epsilon)},
	\end{equation}
and
\begin{equation}\label{eq:nl-3}
e^{-kn_l\epsilon}
\leq
 | \Jac f^{n_l} (y)| \cdot |\Jac f^{n_l} (w)|^{-1} 
\leq e^{kn_l\epsilon}
\quad \mbox{ for every } y,w\in \widetilde U(n_l,x,\kappa).
\end{equation}

\medskip
		
Consider now any $N \geq n(\epsilon)$ and fix $l^\star=l^\star (N)$ such that $n_{l^\star} \leq N < n_{l^\star+1}$.
Such $l^\star$ exists since $n_0\leq n(\epsilon)$ and $n_l\to \infty$ as $l\to \infty$. 
It follows from the definition \eqref{eq:def-M} 
of $M$ that
\begin{equation}\label{eq:inclusion-M-1}
 B
 \big(f^N (z), \kappa\,  e^{-(N-n_{l^\star})M}\big)
\supseteq
f^{N-n_{l^\star}} 
\big(B( f^{n_{l^\star}} (x), \kappa\, e^{-2(N-n_{l^\star})M})\big)
\end{equation}
and 
\begin{equation}\label{eq:inclusion-M-2}
f^{n_{l^\star+1}-N} 
\left( B\big( f^{N} (x), \kappa \, e^{-(n_{l^\star+1}-N) M}\big)\right)
\subseteq
 B(f^{n_{l^\star+1}} (x), \kappa).
\end{equation}	
It follows from \eqref{eq:inclusion-M-2},
the second inequality in \eqref{eq:nl-1} applied with $l=l^\star+1$ and $t=1$,
 and the fact that $0\leq n_{l^\star+1}-N \leq \epsilon n_{l^\star} \leq \epsilon N$, 
  that $U(N, x, \kappa, \epsilon)$ is well-defined and satisfies 
\begin{equation}\label{e:U-n-nl+1}
U(N,x,\kappa, \epsilon)
\subseteq
\widetilde U ( n_{l^\star+1}, x,\kappa)
\subseteq
\mathcal E_x 
 (\kappa \,
 e^{-n_{l^\star+1} (\chi_j-  \epsilon )})
 \subseteq
 \mathcal E_x 
 (\kappa \,
 e^{-N (\chi_j-  \epsilon )}).
\end{equation}
Similarly, from \eqref{eq:inclusion-M-1}, the first inequality in \eqref{eq:nl-1}  applied with
$l=l^\star$ and $t= e^{-2(N-n_{l^\star})M}$, 
and the fact that $0\leq N-n_{l^\star}\leq \epsilon n_{l^\star} \leq \epsilon N$
we deduce that
\[
\begin{aligned}
U(N,x,\kappa,\epsilon)
& \supseteq 
f^{-n_{l^\star}}_{T^{n_{l^\star}} (\hat x)} 
\left(B\big( f^{n_{l^\star}} (x), \kappa\, e^{-2(N-n_{l^\star})M}\big)\right)\\
& \supseteq
\mathcal E_x
(\kappa \,
 e^{-n_{l^\star} (\chi_j- (2M+1)\epsilon )})
 \supseteq
 \mathcal E_x
(\kappa \,
 e^{-N (\chi_j- (2M+1)\epsilon )}),
\end{aligned}\]
which completes the proof of the first item.

\medskip

The second and the third assertions follow
 from similar arguments,
combining
\eqref{eq:inclusion-M-1}
 and \eqref{eq:inclusion-M-2}
 with \eqref{eq:nl-2} and  \eqref{eq:nl-3}, respectively.
\end{proof}

The following corollary records a special case of the above lemma when all the Lyapunov exponents of $\nu\in\mathcal M^+ (f)$ are 
equal.

\begin{cor}\label{c:nu-U-compare-equal}
Assume that all the Lyapunov exponents of $\nu\in\mathcal M^+ (f)$
are equal to $\chi>0$. Then, for all $0<\epsilon \ll \chi$,
$x \in \pi(Z_\nu(\epsilon))$, $N \geq n(\epsilon)$, and $0<\kappa < r(\epsilon)$, we have
\begin{enumerate}
\item $B(x, \kappa\,
e^{-N
(\chi+ (2M+1)\epsilon)}) \subseteq U(N,x,\kappa,\epsilon) \subseteq B(x,  \kappa\,
e^{-N 
(\chi-  \epsilon)})$;
\item $\kappa^{2k} 
e^{-2kN 
(\chi + (2M+1)\epsilon) } \leq \volume (U(N, x, \kappa, \epsilon)) \leq \kappa^{2k}
e^{-2kN 
 (\chi -\epsilon)}$.
\end{enumerate}
\end{cor}

\begin{rmk}\label{rmk:section2-hyperbolic}
If $f$ is hyperbolic, then we have $Z_\nu \cap \pi^{-1} (J(f))= Z\cap \pi^{-1} (J(f))$
 for any $\nu \in \mathcal{M}^+_J(f)$.
Moreover, observe that any ergodic probability measure
on $J(f)$ 
belongs to $\mathcal M^+_J(f)$.
 In particular,
we have
$\pi(Z_\nu)  \supseteq J(f)$ for any ergodic probability measure on $J(f$).
We also have $ Z_{\nu,\alpha}\cap \pi^{-1} (J(f))
= Z_\nu\cap \pi^{-1} (J(f))$
 for all $\alpha$ sufficiently small, which implies that 
we can take
$Z(\epsilon)=Z'(\epsilon)=Z$
for all $\epsilon$ sufficiently small.

More generally, let
$X\subseteq \mathbb P^k$ be
a closed invariant uniformly expanding set. 
Take $\nu\in \mathcal M^+_X (f)$.
Denoting by $O_X$ the set of orbits $\{x_n\}_{n \in \mathbb Z} \in X^{\mathbb Z}$,
it follows from the definition of $\hat \nu$ that $\hat \nu (O_X)=1$
and that we can assume that $X\subseteq \pi(Z_\nu (\epsilon))$
for all $\epsilon>0$.
As $\hat \nu (O_X)=1$,
we can also assume 
 that
 $\pi(Z_\nu (\epsilon))\subseteq X$, hence
 $\pi(Z_\nu (\epsilon)) = X$,
  for all $\epsilon>0$.
\end{rmk}

\subsection{Volume-conformal measures}\label{ss:volume-conformal}
We again fix in this section
a holomorphic endomorphism $f$
of $\mathbb P^k$ of algebraic degree $d\geq 2$, and we let $X$ be a closed invariant set for $f$.
Recall that $f|_{X}: X \to X$ is {\it topologically exact} if for any open set $U\subset \mathbb P^k$ 
with $U\cap X\neq \emptyset$
 there exists $n \ge 1$ such that $f^n(U) \supseteq X$.

\begin{defn}\label{d:conformal}
Given any $t \ge 0$, a probability measure $\mu$ on $X$ 
is {\it $t$-volume-conformal on $X$}
if, for every Borel subset $A\subseteq X$
on which $f$ is invertible, we have
\begin{equation*}
	\mu (f (A)) =  \int_A |\Jac f|^t d\mu.
\end{equation*}
We define
\[\begin{aligned}
	\delta_X (f)  & \defeq \inf \left\{ t \ge 0 \colon  \mbox{there exists a $t$-volume-conformal
		measure on } X\right\}. 
\end{aligned}\]
\end{defn}
 
\begin{lem}\label{l:conformal}
Assume that
 $f|_{X}$ is topologically exact.
	Let $\mu$ be a probability measure on $X$
	which is $t$-volume-conformal on $X$
	 for some $t\ge 0$. Then
	\begin{enumerate}
	\item the support of $\mu$ is equal to $X$;
\item 	for every $r>0$
 there exists
 constants
	$0< m_- = m_-(\mu,r) \leq 1$ 
	and $0<m_+= m_+(\mu,r)\leq 1$
	such that
	$m_- \leq \mu (B(x,r)) \leq m_+$
	for every $x \in X$.
	\end{enumerate}
\end{lem}

\begin{proof}
	Assume that there exists a point $x \in X$
	 which does not belong
	to the support of $\mu$. 
	Take a small ball $B$ centred at $x$
	which is disjoint from the critical set $C(f)$ of $f$ 
	and
	such that
	$\mu(B)=0$. 
	As 
	$f|_{X}$ is topologically exact,
	  we have 
	$X \subseteq  
	 f^n(B)$ for some $n\geq 1$.
	Hence, it is
	 enough to prove that 
	$\mu(f^n(B))=0$ for all $n\in \mathbb N$.	
	Since $B\cap C(f)=\emptyset$, 
	this is a consequence of 
	the volume-conformality of $\mu$ and the fact that $\mu(B) = 0$
	(we need here to partition $B$ into subsets where $f^n$ is injective in order to apply Definition
	\ref{d:conformal}).
	The first assertion follows.

	The second assertion is a consequence of the first	
	 and the fact
that, for every probability measure $\mu$ on $\mathbb P^k$ 
and $r>0$, there exist constants 
$m_{\pm}=m_{\pm}(\mu,r)$
 such that (2) holds for every $x$ in the support of $\mu$.
\end{proof}

Recall that, for every
 $\nu\in \mathcal M^+(f)$ and every
 $\epsilon$ sufficiently small, 
$Z_\nu(\epsilon)$, $r(\epsilon)$, 
and $n(\epsilon)$ are given by Corollary \ref{cor_cordistortion} and Lemma \ref{cor_cordistortion4}, and the sets 
$U(N,x,\kappa,\epsilon)$ are defined in Definition \ref{d:U}.
 
\begin{lem}\label{l:mcm-conformal}
Assume that $f|_{X}$ is topologically exact.
Fix  $\nu \in \mathcal{M}_X^+(f)$, $t\geq 0$, and
$0< \epsilon\ll \chi_{\min}$, where $\chi_{\min}>0$ is the smallest Lyapunov exponent of $\nu$.
 Then, for every $0<\kappa<r(\epsilon)$, every 
 $t$-volume-conformal probability measure $\mu$ on $X$,
 every 
 $x \in \pi(Z_\nu (\epsilon))$, and every 
 $N\geq n(\epsilon)$, the set $U = U(N,x,\kappa,\epsilon)$ is well-defined and
 satisfies 
	\[
	\frac{
m_-(\mu,\kappa\, e^{-MN\epsilon} )
}{
C^{t} \kappa^{tk}
		}e^{-tNk(5M+2)\epsilon} \leq
		\frac{ \mu  (U)}{ \volume (U ) ^{t/2}}
		 \leq  
\frac{C^{t}m_+(\mu,\kappa\, e^{-MN\epsilon} )
}{
\kappa^{tk}}  e^{tNk(5M+2)\epsilon},
	\]	
where
$M$ is as in \eqref{eq:def-M},
the constants $m_-$ and $m_+$ are as in Lemma \ref{l:conformal},
and $C$ is a positive constant 
 independent of $\kappa$, $\epsilon$, $x$, $N$, $\nu$, $\mu$, and $t$.
\end{lem}

\begin{proof}
Fix $x \in \pi(Z(\epsilon))$, $N \ge n(\epsilon)$, and $0<\kappa < r(\epsilon)$. The set $U \defeq U(N,x, \kappa,\epsilon)$
is well-defined by Lemma \ref{l:nu-U-compare}.
We denote for simplicity by $B$ the ball $B(f^N(x), \kappa \, e^{-MN\epsilon})= f^N(U).$

\medskip

Let $\mu$ be any $t$-volume-conformal probability measure on $X$.
Since $X$ has a dense orbit, by Lemma \ref{l:conformal}
the support of $\mu$ is equal to $X$.
Since $f^N$ is injective on $U$, by Definition \ref{d:conformal}
we have
	\[
\mu(B)
=
\mu(f^N (U))	
=
\mu (f^N(U \cap X)) = \int_{U \cap X} |\Jac f^{N}|^t d\mu
=
 \int_{U} |\Jac f^{N}|^t d\mu.
	\]
We deduce from Lemma \ref{l:nu-U-compare} (3)
	 that 
	\[
	e^{-tkN(2M+1)\epsilon }  |\Jac f^{N} (x)|^t\mu(U )
	\leq \mu  (B) \leq 
	e^{tkN(2M+1)\epsilon} |\Jac f^{N} (x)|^t \mu (U ).
	\]
It follows from the above expression
	 that
	\begin{equation}\label{e:Uconf1}
	\frac{e^{-tkN (2M+1)\epsilon}}{ |\Jac f^{N} (x)|^t} 
	\cdot m_-
	\leq \mu(U) \leq \frac{e^{tkN (2M+1)\epsilon}}{  |\Jac f^{N} (x)|^t}
	\cdot m_+,
	\end{equation}
	where $m_-\defeq m_- (\mu,\kappa  \, e^{-MN\epsilon} )$
	and $m_+\defeq m_+(\mu,\kappa  \, e^{-MN\epsilon} )$ are as in Lemma \ref{l:conformal}.

	Again by 
	Lemma \ref{l:nu-U-compare} (3),
	 we also have
	\begin{equation}\label{e:Uconf2}
	\frac{e^{-2 kN (2M+1)\epsilon}\volume (B)
	}{  |\Jac f^{N} ( x)|^2} 
	\leq 
	\volume(U)=
	\int_{B (f^{N} (x), \kappa \, e^{-MN\epsilon})}  |\Jac f^{-N}_{\hat z}|^2
	\leq \frac{ e^{2kN (2M+1)\epsilon} 
	\volume (B)
	  }{ |\Jac f^{N} (x)|^2},
	\end{equation}
	where
	the integral is taken with respect to the Fubini-Study metric,
$\hat z$ is any element in $Z$ such that $z_0 = f^{N}(x)$ and $z_{-N} = x$,
	and we observe that $f^{-N}_{\hat z}$ is well-defined on $f^N(U)$
 by Lemma \ref{l:nu-U-compare}.
  	
\medskip

Combining the inequalities \eqref{e:Uconf1} and \eqref{e:Uconf2}, we see that
\[
\frac{m_-}{\volume (B)^{t/2}}e^{-t Nk(4M+2)\epsilon} \volume (U)^{t/2} \leq \mu(U) \leq \frac{m_+}{\volume (B)^{t/2}} e^{tNk(4M+2)\epsilon}\volume (U)^{t/2}.
\]
The assertion follows from the last expression by observing that there exists a positive constant $C$ such that $C^{-2} \leq \volume (B (x,r)) / 
r^{2k}\leq C^2$ for every $x\in \mathbb P^k$ and $0<r<1$.
\end{proof}

\subsection{A pressure for expanding measures}\label{ss:pressure-exp}
Let $f$ be a holomorphic endomorphism of $\mathbb P^k$
of algebraic degree $d\geq 2$. For any invariant probability measure $\nu$ and $t\in \mathbb R$, we define
$$P_\nu(t) \defeq h_\nu(f) -t \int |\Jac f| d\nu = h_\nu(f) -t L_{\nu}(f).$$
Let $X\subseteq \mathbb P^k$ be a closed invariant set for $f$. 
We define a pressure function $P^+_X$ as
\begin{equation}\label{eq:def:PX+}
\begin{aligned}
P^+_X (t)
 & \defeq \sup \big\{P_\nu(t) \colon \nu \in \mathcal{M}_X^+(f)\big\}
\end{aligned}
\end{equation}
 and set
\[
p^+_X(f) \defeq \inf  \big\{t \colon  P_X^+(t)=0\big\}.
\]
We will drop the index $X$ when $X=\bbP^k$.

\begin{lem} \label{lemma_pressure}
Let $X\subseteq \mathbb P^k$ be a closed invariant set
for $f$.
Assume that $\mathcal M_X^+(f)$ is not empty.
Then
 we have $P_X^+ (t)< \infty$
 for all $t \in \mathbb R$
  and  the function $t\mapsto P_X^+(t)$ is convex and non-increasing.
\end{lem}

\begin{proof}
Take $\nu\in\mathcal M_X^+(f)$.
As $L_\nu(f)>0$ and the topological entropy of $f$ is bounded by $k \log d$
\cite{Gromov03,dinh2010dynamics}
 we have
$P^+_X (t) \leq k \log d$ for all $t\geq 0$.  Take now $t<0$. Since the function $|\Jac f|$
is bounded from above by a constant $M'$
and $h_\nu (f)\leq k\log d$,
 we have 
$P^+_X (t) \leq k\log d + |t| M'$ for
every
 $t<0$.  Hence, $P^+_X (t)<\infty$ for all $t \in \mathbb R$.

For any given measure $\nu \in \mathcal M_X^+(f)$, the function $t\mapsto P_\nu (t)$ is non-increasing. It follows from 
its definition
\eqref{eq:def:PX+}
that the function $t\mapsto P_X^+(t)$ is non-increasing.
 It is convex as it is a supremum of affine, hence convex, functions.
\end{proof}

The following example illustrates that 
Lemma \ref{lemma_pressure} 
is false
(even with $X=J(f)$)
 if we take the supremum over the set of all ergodic probability
 measures, with no requirement on the Lyapunov exponents,
  in the definition 
  \eqref{eq:def:PX+}
  of the pressure function $P_X^+(t)$.

\begin{example}
It is possible to construct endomorphisms $f$ of $\mathbb P^k$ admitting a saddle fixed point $p_0$ in the Julia set and
with
 $|\Jac f (p_0)|<1$
(and actually also equal to $0$). 
An example of this phenomenon is given for instance by Jonsson in \cite[Example 9.1]{Jonsson-skew},
see also \cite[Theorem 6.3]{BDM07}, \cite{Taflin10}, and \cite[Remark 2.6]{BT17} for further examples.
Consider the polynomial self-map $f$ 
of $\mathbb C^2$
defined as
\[
(z,w)\mapsto \left(z^2, w^2 + 2(1+ \eta-z)w\right),
\]
which extends to $\mathbb P^2$ as a holomorphic endomorphism. 
As $f$ preserves the families of the vertical lines parallel 
to 
$\{z=0\}$, for every $(z_0,w_0)\in \mathbb C^2$
 the \emph{vertical eigenvalue} of $Df_{(z_0,w_0)}$
  is well-defined.
It is immediate to check that,
for $0\leq \eta < 1/2$, the point
$p_0 =(1,0)$ is a saddle fixed point, with vertical eigenvalue equal to $2\eta$, and Jacobian equal to $4\eta$. In particular,
the Jacobian of $f$ at $p_0$ 
can take any small non-negative value (including 0).
The point $p_0$ is in $J(f)$ since $J(f)$ is closed and, 
for Lebesgue almost all $z_0 \in S^1$, the point $(z_0,0)$ belongs to $J(f)$. This  
follows from a direct computation of the derivatives
 which, by Birkhoff's ergodic theorem, gives that
\[
Df^n_{(z_0,0)} \sim 
\begin{pmatrix}
2^n & 0 \\
\star & \int_{0}^{2\pi} \log |1+\eta -e^{i \theta}| d\theta
\end{pmatrix}
= \begin{pmatrix}
2^n & 0 \\
\star & 2^n  \log |1+\eta|
\end{pmatrix},
\]
and the characterization of the Julia set of $f$ given in \cite[Corollary 4.4]{Jonsson-skew}.

Consider the function $$P_J(t) \defeq \sup_\nu P_\nu(t) $$ where now the supremum is taken over the set of {\it all} invariant 
probability
measures supported on $J(f)$. 
If $\nu_0 = \delta_{p_0}$ is the
Dirac mass
at $p_0$, then the function
$t \mapsto P_{\nu_0}(t)$ is increasing in $t$ and $P_{\nu_0}(0)=0$. Hence, for such 
an endomorphism $f$, the 
function $P_J(t)$ is convex but it increases after some $t_0>0$ and has no zeroes.
\end{example}

\begin{rmk}
One could define $\widetilde{P}_J^+(t)$ by considering the set of
all
ergodic
probability measures with positive {\it sum} of Lyapunov exponents in the definition of $P_J^+(t)$. However, it is
unclear to us how to  generalize many of the results in this paper,
and in particular 
Theorem \ref{thm_main}, 
to this larger class of measures. A priori, it could be possible that the first zero of $P_J(t)$
is larger than the first zero of $P_J^+(t)$, 
but (possibly) equal to the first zero of $\widetilde{P}_J^+(t)$.
\end{rmk}

\section{Exact volume dimension of measures in $\mathcal{M}^+(f)$}  \label{sec_proofmain}
Let $f: \bbP^k \to \bbP^k$ be a holomorphic endomorphism of algebraic degree $d \ge 2$.
In this section we define a pointwise dynamical volume dimension for every measure
$\nu \in \mathcal{M}^+(f)$ and prove that it is constant $\nu$-almost everywhere.

\medskip

Fix a measure $\nu \in \mathcal{M}^+(f)$ and let $\chi_{\min}>0$ be the
smallest Lyapunov exponent of $\nu$. For every $0<\epsilon\ll \chi_{\min}$,
we fix $Z_\nu(\epsilon)$, $n(\epsilon)$, and $r(\epsilon)$ as given by Corollary \ref{cor_cordistortion}
and Lemma \ref{cor_cordistortion4}. For every $x \in \pi(Z_\nu(\epsilon))$,
the sequence $\{n_l\}_{l\geq 0}= \{n_l (x)\}_{l\geq 0}$ is also given by Lemma \ref{cor_cordistortion4};
see Remark \ref{r:nl-z}. For $x \in \pi(Z_\nu(\epsilon))$, $0<\kappa<r(\epsilon)$, and $N\geq n(\epsilon)$, we define
\begin{equation}\label{eq:def-delta-xekN}
	\delta_x (\epsilon,\kappa, N) \defeq \frac{\log \nu(U(N,x,\kappa,\epsilon))}{\log {\rm Vol}(U(N,x,\kappa,\epsilon))},
\end{equation}
where $U(N,x,\kappa,\epsilon)$ is as in Definition \ref{d:U}. Observe that, for every $\epsilon$, $x$, $\kappa$, 
and $N$ as above, the definition  of $\delta_x (\epsilon, \kappa, N)$
is well-posed by Lemma \ref{l:nu-U-compare}.

\medskip

Recall that the set $Z_\nu$ (see Definition \ref{d:Z})
 satisfies
$Z_\nu = \cup_{\epsilon>0} Z_\nu (\epsilon)$ up to a $\nu$-negligible set,
and that the family
 $\{Z_\nu (\epsilon)\}_{\epsilon>0}$
  is non-decreasing as $\epsilon\to 0$. 
In particular, $\nu$-almost every $x \in \pi(Z_\nu)$ belongs to $\pi(Z_\nu (\epsilon))$ for every 
$0<\epsilon<\epsilon_0$ for some $\epsilon_0=\epsilon_0 (x)$.
For every such $x$, 
we define the {\it upper}
and the {\it lower local volume dimension} at $x$ as
\begin{equation}\label{eq_localdim}
\overline{\delta}_x \defeq \limsup_{\epsilon \to 0} \limsup_{\kappa \to 0} 
\limsup_{N \to \infty}\delta_x (\epsilon,\kappa, N)
\quad 
\mbox{ and }
\quad
\underline{\delta}_x \defeq \liminf_{\epsilon \to 0} \liminf_{\kappa \to 0}
\liminf_{N \to \infty} \delta_x(\epsilon,\kappa,N),
\end{equation}
respectively, where 
$\delta_{x} (\epsilon, \kappa, N)$
is as in \eqref{eq:def-delta-xekN}.
\begin{defn}\label{d:local-dim}
	If $\underline{\delta}_x= \overline{\delta}_x$, 
	we say that $\delta_x \defeq \underline{\delta}_x  = \overline{\delta}_x$ is the {\it local volume dimension} 
	of $\nu$ at $x$. We say that $\nu \in \mathcal{M}^+(f)$ is
	 {\it exact volume-dimensional} if the local volume dimension $\delta_x$ exists for $\nu$-almost every $x$.
\end{defn}

The main result of this section is the following theorem. Recall that $h_\nu(f)$, $L_\nu(f)$,
and $\chi_{\min}$ denote the measure-theoretic entropy, the sum of the
Lyapunov exponents, and the smallest Lyapunov exponent
 of $\nu$, respectively.

\begin{thm}\label{thm_11.4.2}
	Let $f: \bbP^k \to \bbP^k$ be a holomorphic endomorphism of algebraic degree $d \ge 2$.
	Take $\nu\in \mathcal M^+(f)$ and $0<\epsilon\ll \chi_{\min}$.
	Then, for $\nu$-almost all $x \in \pi(Z_\nu(\epsilon))$ and all $0 < \kappa < r(\epsilon)$, there exists
	integers $m_1(\epsilon, x)\geq n(\epsilon)$ and $m_2 (\epsilon,\kappa)\geq 0$
	 such that
	\[
	\frac{h_{\nu}(f)}{2L_{\nu}(f)}
	-c \epsilon
	\leq
	\delta_x (\epsilon, \kappa, N)
	\le
	\frac{h_{\nu}(f)}{2L_{\nu}(f)} + c\epsilon
	\quad \mbox{  for all  } N\geq m_1(\epsilon, x) + m_2 (\epsilon,\kappa),\]
	where 
	$\delta_{x} (\epsilon, \kappa, N)$ is as in \eqref{eq:def-delta-xekN}
	and  $c > 0$  is a constant 
	independent of $\epsilon$, $x$, and $\kappa$.
\end{thm}

\begin{rmk}\label{r:1142}
Although Theorem \ref{thm_11.4.2} is stated for points  $x\in \pi(Z_\nu(\epsilon))$,
we can associate to $\hat \nu$-almost every $\hat x\in Z_\nu(\epsilon)$
the integer $m_1(\epsilon, \hat x)\defeq m_1(\epsilon, x_0)$, where, since we have $x_0\in \pi(Z_\nu(\epsilon))$,
the number $m_1(\epsilon, x_0)$ is given by Theorem \ref{thm_11.4.2}. 
\end{rmk}

The following consequence of Theorem \ref{thm_11.4.2}
shows that every $\nu \in \mathcal{M}^+(f)$ is exact volume-dimensional.

\begin{cor} \label{cor_exactdim}
	Let $f: \bbP^k \to \bbP^k$ be an endomorphism of algebraic degree $d \ge 2$
	and take $\nu \in \mathcal M^+ (f).$ For $\nu$-almost every $x \in\mathbb P^k$,
	the local volume dimension $\delta_x$ is well-defined and equal to $(2L_\nu (f))^{-1} h_\nu (f)$.	
\end{cor}

\begin{proof}
	Recall that  the family
	 $\{Z_\nu(\epsilon)\}_{\epsilon>0}$
	  is non-decreasing for $\epsilon\to 0$, and that
	we have
	$Z_\nu = \cup_{\epsilon>0} Z_\nu (\epsilon)$ up to a $\nu$-negligible set.
	In particular, for
	$\nu$-almost every
	$x \in \mathbb P^k$  there exists $\epsilon_0 = \epsilon (x_0)>0$ such that
	$x$ belongs to $\pi(Z_\nu (\epsilon))$ for every $0<\epsilon<\epsilon_0$.
	The assertion follows from the definition
	\eqref{eq_localdim}
	of the upper and lower volume dimensions
	 and Theorem
	\ref{thm_11.4.2}.
\end{proof}

The rest of the section is devoted to the proof of Theorem \ref{thm_11.4.2}. 
We will follow the
 general
  strategy presented in
 \cite[Section 11.4]{PU} 
but we will need to use the results
in Section \ref{subsec_dist} to replace the distortion estimates for univalent maps in dimension $1$.

\subsection{Proof of Theorem \ref{thm_11.4.2}: a reduction}
Fix a countable measurable partition $\mathcal P$ of $\mathbb P^k$.
Up to taking the elements of the partition sufficiently small,
we can assume that the \emph{entropy}
$h_{\nu}(f,\mathcal P) $ of the partition $\mathcal P$ satisfies
$h_{\nu}(f)-\epsilon \leq h_{\nu}(f,\mathcal P) \leq h_{\nu}(f)$.
Recall that, by the 
Shannon-McMillan-Breiman Theorem
\cite{Parry69,W00}
for $\nu$-almost every $x \in \mathbb P^k$
we have
$$\lim_{n \to \infty} -\frac{1}{n} \log \nu(\mathcal{P}^{n} (x)  ) =: h_{\nu}(f,\mathcal{P}).$$
Here $\mathcal P^n$ is the partition generated by $\mathcal P, f^{-1} \mathcal P, \dots, f^{-n} \mathcal P$
(i.e., the partition whose elements are the sets of the form
$P_0 \cap f^{-1}(P_1)\cap \ldots \cap f^{-n} (P_{n})$ for $P_0, \dots, P_{n}\in \mathcal P$), and
$\mathcal P^n (x)$ denotes the element of the partition $\mathcal P^n$ containing $x$.

\begin{prop}\label{p:double-ellipses}
	Fix $\nu\in \mathcal M^+(f)$. For every $0<\epsilon\ll \chi_{\min}$
	there exist two
	partitions $\mathcal P_1$ and $\mathcal P_2$ 
	with
	$h_{\nu}(f,\mathcal P_1)\geq h_\nu(f) -\epsilon$
	and four constants
	$b_E,b_F$ (independent of $\epsilon$)
	and
	 $c_E, c_F>0$ (possibly depending on $\epsilon$)
	such that for $\nu$-almost every $x \in \pi (Z_\nu(\epsilon))$
	there 
	exists an integer 
	$m(\epsilon, x)\geq n(\epsilon)$ 
	such that
	for all $n\geq m(\epsilon, x)$, we have
	\[
	E(n) \defeq \mathcal E_x (c_E e^{-n (\chi_j+b_E \epsilon)})
	\subseteq 
	\mathcal P_1^n (x) 
	\quad
	\text{ and }
	\quad
	\mathcal P_2^n (x)
	\subseteq
	F(n) \defeq \mathcal E_x (c_F e^{-n (\chi_j- b_F \epsilon)}).
	\]
\end{prop}

We prove the existence of the sequences $E(n)$ and $F(n)$ and partitions
$\mathcal{P}_1$ and $\mathcal{P}_2$ in the next two subsections.
We now show how Theorem \ref{thm_11.4.2} is a consequence of Proposition \ref{p:double-ellipses}.

\begin{proof}[Proof of Theorem \ref{thm_11.4.2} assuming Proposition \ref{p:double-ellipses}]
	We fix $\epsilon>0$ as in the statement, 
	$x\in \pi (Z_\nu(\epsilon))$, and
	$0<\kappa<r(\epsilon)$.
	For every $N \ge n(\epsilon)$
	and $0<\kappa<r(\epsilon)$,
	define the
	 integers $n_E (N, \kappa)$
	and $n_F(N,\kappa)$ 
	as
	\[
	n_E (N,\kappa)
	\defeq 
	\min_j \Big\lfloor
	\frac{
		(\chi_j-\epsilon)N + \log c_E - \log \kappa}{\chi_j+b_E\epsilon}
	\Big\rfloor \]
	and
	\[n_F(N,\kappa)
	\defeq\max_j
	\Big\lceil
	\frac{[\chi_j+(2M+1)\epsilon)]N + \log c_F - \log \kappa}{\chi_j-b_F \epsilon}
	\Big\rceil, \]
	where $b_E,b_F,c_E, c_F$
		are as in Proposition \ref{p:double-ellipses} and we recall that the $\chi_j$'s 
	are the Lyapunov exponents of $\nu$, which are strictly positive.
	Then, by Lemma \ref{l:nu-U-compare} (1) and (2)
	and Proposition \ref{p:double-ellipses}, for all $0<\kappa<r(\epsilon)$, we have
	\begin{equation}\label{e:4241-1}
		\mathcal P^{n_F (N,\kappa)}_2 (x) \subseteq F(n_F(N,\kappa))
		\subseteq 
		U(N,x,\kappa,\epsilon) 
		\subseteq
		E(n_E (N,\kappa)) \subseteq \mathcal P^{n_E(N,\kappa)}_1 (x)
		 \mbox{  for all } N \ge  m(\epsilon,x),
	\end{equation}
	where $m(\epsilon,x)$ is as in Proposition \ref{p:double-ellipses},
	and
	\begin{equation}\label{e:4241-2}
		\kappa^{2k}e^{-2N(L_{\nu}+k(2M+1)\epsilon)} \le \volume(U(N,x,\kappa,\epsilon)) \le
		\kappa^{2k}e^{-2N(L_{\nu}-k\epsilon)}
		\quad \mbox{ for all } N\geq n(\epsilon),
	\end{equation}
	where we recall that 
	$M$ is as in \eqref{eq:def-M}.
	
	\medskip

It follows from \eqref{e:4241-1} and the Shannon-McMillan-Breiman Theorem that there exists
$m'(\epsilon,x)\geq m(\epsilon,x)$ and 
 $m''(\epsilon,\kappa) \gg 1$ such that
\begin{equation}\label{e:4241-num}
(h_\nu(f)-2\epsilon) \Big(
		\lim_{N\to \infty}
		\frac{n_E (N,\kappa) }{N}-\epsilon\Big)
		\leq
		-\frac{\log \nu (U(N,x,\kappa,\epsilon))}{N}
		\leq
		(h_\nu (f)+\epsilon) \Big( \lim_{N\to \infty} \frac{n_F(N,\kappa) }{N}+\epsilon\Big)
\end{equation}
for all $N > m'(\epsilon,x) +  m''(\epsilon,\kappa)$. 
We used here the fact that, since $\kappa<r(\epsilon)$, the integers
 $n_E(N,\kappa)$ and $n_F(N,\kappa)$ are bounded below by quantities
 which are independent of $0<\kappa<r(\epsilon)$.
 Similarly,
 it follows from \eqref{e:4241-2} that there exists $m'''(\epsilon,\kappa) \gg 1$
such that
\begin{equation}\label{e:4241-den}
		2(L_\nu(f) - k\epsilon) - \epsilon
		\leq
		-\frac{\log \volume(U(N,x,\kappa,\epsilon))}{N}
		\leq
		2(L_\nu(f) + (2M+1) \epsilon) + \epsilon
\end{equation}
for all $N > m'''(\epsilon,\kappa) +  n(\epsilon)$.

\medskip

 Setting
\[m_1 (\epsilon,x) \defeq  m'(\epsilon,x) \geq m(\epsilon, x) \geq n(\epsilon)
\quad \mbox{ and } \quad
m_2 (\epsilon,\kappa)\defeq \max \{m''(\epsilon,\kappa), m'''(\epsilon,\kappa)\},\] 
the assertion follows combining \eqref{e:4241-num}, \eqref{e:4241-den},
and the definitions of $n_E (N,\kappa)$ and $n_F (N,\kappa)$.
\end{proof}

\subsection{Proof of Proposition \ref{p:double-ellipses}: the existence of $E(n)$ and $\mathcal{P}_1$}
\label{ss:En}
We will need the following lemma, see for instance
\cite[Corollary 9.1.10]{PU}.

\begin{lem}\label{lem_9.1.10}
Let 
$(\mathcal X,\delta)$
be
 a compact metric space, 
	$f : \mathcal X \to \mathcal X$ 
	a measurable map with respect to the Borel $\sigma$-algebra on $\mathcal X$ 
	and $\nu$ an $f$-invariant 
	Borel probability measure.
	 Then for every $r>0$, there exists
	$\mathcal{X}_0\subseteq \mathcal X$ with $\nu(\mathcal{X}_0)=1$ and
	a finite partition $\mathcal{P}$
	of $\mathcal X$ into Borel sets of positive measure $\nu$ and 
	of diameter
	smaller than
	$r$ such that,
	 for every $\epsilon > 0$ and
	every $x \in \mathcal{X}_0$,
	there exists an integer $m_0=m_0(\epsilon, x)$ 
	such that 
	$$B_{\mathcal X}(f^n(x), e^{-n\epsilon}) \subset \mathcal{P}(f^n(x))
	\quad \mbox{ for every } n \ge m_0,$$
	where $B_{\mathcal X} (y,a)$ denotes the open ball in $\mathcal X$
	of radius $a$ and center $y\in \mathcal X$.
\end{lem}

Fix $0<\epsilon\ll \chi_{\min}$.
Let $\mathcal{X}_0$,
$\mathcal{P}$,
and $m_0(\epsilon, x)$ 
be as given by Lemma \ref{lem_9.1.10} applied with $\mathcal{X} = \text{Supp } \nu$.
Up to taking $r$ sufficiently small, we can assume that $h_{\nu}(f,\mathcal P) \geq h_\nu(f) -\epsilon$.
Up to replacing $Z_\nu(\epsilon)$ with $\pi^{-1}(\mathcal{X}_0) \cap Z_\nu(\epsilon)$,
we can assume that the conclusion of Lemma \ref{lem_9.1.10} holds for all $x \in \pi (Z_\nu(\epsilon))$. 

\medskip

Fix $x \in \pi(Z_\nu(\epsilon))$. In particular, there exists $\hat x \in Z_\nu(\epsilon)$ with $\pi(\hat x)=x$.  
Let $\{n_l\}_{l\geq 0}$ be the sequence associated to
$\hat x$
by  Lemma \ref{cor_cordistortion4}. 
We fix $l_0\in \mathbb N$
such that $n_{l_0} \geq m_0(\epsilon, x)$.
Recall that $M$ is as in \eqref{eq:def-M}.

\medskip

Consider an integer $n \ge n_{l_0}$ and the dynamical ellipse 
$$E (n) \defeq \mathcal E_x\big( Cr(\epsilon)e^{-n(\chi_{j}  + (M+2 )\epsilon)}\big),$$
where $0<C<1$ is a constant small enough so that
\begin{equation}\label{eq_3}
	f^q(E (n)) \subset \mathcal{P}(f^q(x))\quad
	\mbox{ for every } q \le n_{l_0}.
\end{equation}

We now show that $f^q(E(n)) \subset \mathcal{P}(f^q(x))$ for all $n_{l_0} \leq q \leq n$.
To this end, fix one such  $q$ and let
$l^\star =l^\star (q)\geq l_0$
  be such that $n_{l^\star} \le q < n_{l^\star+1}$.
Since $T^{n_{l^\star}}(\hat{x}) \in Z'_\nu(\epsilon)$ by Lemma \ref{cor_cordistortion4}
and $\pi(T^{n_{l^\star}}(\hat{x})) = f^{n_{l^\star}}(x)$, Theorem \ref{thm_distortion}
and Corollary \ref{cor_cordistortion} (3c) imply that there exists a holomorphic inverse branch
$g_{l^\star} \defeq f_{\widehat {f^{n_{l^\star}} (x)}}^{-n_{l^\star}} : B(f^{n_{l^\star}}(x),r(\epsilon)) \to \bbP^k$
of $f^{n_{l^\star}}$ such that 
$g_{l^\star}(f^{n_{l^\star}}(x)) = x$
and 
$$
\mathcal 
E_x \big(r(\epsilon)e^{-n_{l^\star} (\chi_{j} + \epsilon)}\big)
\subseteq
g_{l^\star}
\big(B(f^{n_{l^\star}}(x), r(\epsilon))\big).$$
Set $$E' (n) \defeq \mathcal E_x \big(r(\epsilon)e^{- n_{l^\star} (\chi_{j} + \epsilon)}\big).$$
Then $E(n) \subset E'(n)$ (by the choice of $C$ and the inequality $n \ge n_{l^\star}$)
and Corollary \ref{cor_cordistortion} (3c)
gives
\[
f^{n_{l^\star}}(E(n)) =
f^{n_{l^\star}}\Big( \mathcal E_x \big( Cr(\epsilon)e^{-n(\chi_{j} + (M+2)
	\epsilon)}\big)\Big)
\subseteq
\mathcal E_{f^{n_{l^\star}} (x)} \left(
C
r(\epsilon)e^{-\chi_{j}(n-n_{l^\star})}e^{\epsilon n_{l^\star}-(M+2)\epsilon n} \right).
\]

Since $n \ge n_{l^\star}$, $0 \le q-n_{l^\star} \le \epsilon n_{l^\star}$ (by the definition of the sequence $\{n_l\}_{l\geq 0}$ in Lemma
\ref{cor_cordistortion4}), $0<Cr(\epsilon)<1$ (as $0< r(\epsilon) <1$ by Corollary \ref{cor_cordistortion}),
$q \le n$, 
$q<n_{l^\star+1}$,
and all the $\chi_j$'s are strictly positive, 
by the definition \eqref{eq:def-M} of $M$ and the above expression we deduce that 
\begin{align*}
	f^q(E(n))
	 &
	= f^{q-n_{l^\star}} (f^{n_{l^\star}} (E(n)))
		 \subseteq 
	\mathcal E_{f^q(x)}
	\left(C
	r(\epsilon)e^{-\chi_{j}(n-n_{l^\star})}e^{\epsilon n_{l^\star} - (M+2)\epsilon n}e^{(q-n_{l^\star})M}\right)\\
	&\subseteq 
	 B \left(f^q(x), e^{\epsilon n_{l^\star}}  e^{-2 \epsilon n  -  M \epsilon n} e^{\epsilon n_{l^\star} M}
	\right)
	=
	 B\left(f^q(x),
	e^{\epsilon(n_{l^\star}-n) M +\epsilon(n_{l^\star}-n)-\epsilon n}\right)\\
	& \subseteq B\left(f^q(x),
	e^{-\epsilon n}\right)
	\subseteq B\left(f^q(x), e^{-\epsilon q}\right).
\end{align*}

As $q\geq n_{l_0} \geq m_0(\epsilon, x)$,
by Lemma \ref{lem_9.1.10}
we have $B\left(f^q(x), e^{-\epsilon q}\right) \subset \mathcal{P}(f^q(x))$. It follows that
$f^q(E(n)) \subset \mathcal{P}(f^q(x))$ for all $n_{l_0} \le q \le n$.
 Together with (\ref{eq_3}),  setting $\mathcal P_1 \defeq \mathcal P$
this inclusion implies that
$E (n) 
\subseteq \mathcal P_1^n (x),$
as desired.

\subsection{Proof of Proposition \ref{p:double-ellipses}: the existence of $F(n)$ and $\mathcal{P}_2$}
We work with the same setting and notations as in Section \ref{ss:En}.
We need the following lemma, see for instance \cite[Lemma 11.3.2]{PU}.
Recall that the entropy of a countable partition $\mathcal P= \{P_i\}$
with respect to a probability measure $\mu$ is defined as
\[
H_\mu (\mathcal P) \defeq \sum_{i} -\mu (P_i)\log (\mu (P_i)). 
\]
\begin{lem} \label{l:mane-partition}
	Let $\mu$ be a Borel probability measure 
	on a bounded subset $A$ of a Euclidean space, and $\rho : A \to (0,1]$ 
	a measurable function such that $\log \rho$ is integrable with respect to $\mu$.
	There exists a countable measurable partition
	$\mathcal{P}$ of $A$ such that $H_{\mu}(\mathcal{P})<\infty$
	and 
	$${\rm diam}(\mathcal{P}(x)) \le \rho(x)
	\quad \mbox{ for } \mu\mbox{-almost every } x \in A.$$
\end{lem} 

Recall that $C(f)$ denotes the critical set of $f$
 and that $\nu(C(f))=0$.
For $x \notin C(f)$, define the function 
\begin{equation}\label{eq:rho-F}
	\rho (x) \defeq  c\, r(\epsilon) \min \big\{ {1},|\Jac f| \big\},
\end{equation}
where $c<1$ is sufficiently small so that $f$ is injective on the ball $B(x, \rho (x))$
for every $x \in \mathbb P^k\setminus C(f)$.
 Such constant exists because, since the function $\Jac f (x)$ is  holomorphic in $x$,
 there exists a positive constant $c_0$ 
 such that $|\Jac f (x)|\leq c_0 \cdot  \dist (x, C(f))$  for every $x \in \mathbb P^k$.
For the same reason, the  function $\log \rho$ is integrable with respect to $\nu$,
since by assumption the Lyapunov exponents of $\nu$ are not equal to $-\infty$,
hence $\log |\Jac f|$ is integrable with respect to $\nu$.

Consider a partition $\mathcal P$ given by Lemma \ref{l:mane-partition}, applied with
$\mu=\nu$, $A= \text{Supp } \nu$, and the function $\rho$ 
as in \eqref{eq:rho-F}.
In particular, for $\nu$-almost every $x \in \text{Supp }\nu$ we have $\mathcal P(x)\subset B(x,\rho(x))$.
For every $n \ge 1$, define
\[
V_n (x,\rho) \defeq \bigcap_{j=0}^{n-1} f^{-j} B (f^j (x), \rho (f^j (x))).
\]
It follows from the definition of $\rho$ that $f$ is injective on 
$B (f^j (x), \rho (f^j (x)))$ for all $x \in \mathbb P^k$ and  $0\leq j\leq n-1$. 
As a consequence, 
for every $n \ge 1$ and for $\nu$-almost every $x \in \text{Supp }\nu$,
the map
$f^n$ is injective on $V_n (x,\rho)$ and 
$\mathcal P^n (x)\subset V_n (x,\rho)$.
It is then enough to show that,
 for every $n > n(\epsilon)$,
the set $V_n (x,\rho)$ is contained in
a set
$F(n) \defeq \mathcal E_x (c_F e^{-n (\chi_j- b_F \epsilon)} )$,
for some $b_F$ and $c_F$ 
as in the statement of Proposition \ref{p:double-ellipses}.

Let $l^\star$ be the largest index of the sequence $\{n_l\}_{l\geq 0}$
given by Lemma \ref{cor_cordistortion4}
such that $n_{l^\star} \leq n-1$ (such $l^\star$ exists since $n> n(\epsilon)$).
As in Section \ref{ss:En}, set
$g_{l^\star} \defeq f^{-n_{l^\star}}_{\widehat {f^{n_{l^\star}}(x)}}$.
Then $g_{l^\star}$ is well-defined on $B(f^{n_{l^\star}} (x), r(\epsilon))$.
By the above, and in particular by the injectivity
of $f$ on $V_n$,
we have
\[
V_n (x,\rho) \subset g_{l^\star}
\big( B(f^{n_{l^\star}}(x), \rho(f^{n_{l^\star}} (x))) \big).
\]
By Corollary \ref{cor_cordistortion} (3c),
we deduce that
\[
V_n (x,\rho)\subset \mathcal E_{x} \big( K e^{-n_{l^\star} (\chi_j -\epsilon)}\big)
\]
for some constant $K>0$ independent of $x$ and $n$.
 Since $n-1\leq n_{l^\star+1} \leq  (1+\epsilon)n_{l^\star}$, we deduce that
\[
V_n (x,\rho)\subset F'(n) \defeq
\mathcal E_{x}
\big( K e^{-(n-1) (1+\epsilon)^{-1} (\chi_j- \epsilon)} \big).
\]
Set
$F(n) \defeq
\mathcal E_{x}
\big( c_F e^{-n (\chi_j- b_F\epsilon)} \big)$,
where $c_F\defeq K$ and 
$b_F\defeq (1+\epsilon)^{-1}
\min_j (\chi_j+1)$.
The assertion follows.

\medskip

This concludes the proof of Proposition \ref{p:double-ellipses} and therefore also the proof
 of Theorem
\ref{thm_11.4.2}.

\section{Volume dimension of measures in $\mathcal M^+(f)$} \label{sec_voldim}
Let $f \colon \bbP^k \to \bbP^k$ be a holomorphic endomorphism of algebraic degree $d \ge 2$. The 
goal of this section is to define volume dimensions for sets and measures and study their properties.
More specifically, in Sections \ref{sec_vdimsets} and \ref{subsec_eqdefnl} 
we define and study the {\it volume dimension}
$\VD(\nu)$
 for measures $\nu \in \mathcal M^+(f)$ and $\VD_\nu (X)$ for subsets $X$ of the support of $\nu$. 
In Section \ref{subsec_keyprop} we prove a criterion to relate
 the volume dimension 
of a set of positive measure 
to the
local volume dimensions 
defined in Section \ref{sec_proofmain}; see Proposition \ref{prop_2.1}.
 This criterion, 
 together with Theorem \ref{thm_11.4.2}, will allow us to prove Theorem \ref{thm_main} in the next section.

\subsection{Definition of volume dimension and first properties} \label{sec_vdimsets}
Given
 $\epsilon>0$, 
$\kappa >0$,
$N\in \mathbb N$, and $W\subseteq \bbP^k$,
 we consider the collection $\mathcal U^{\kappa}_N(W, \epsilon)$
  of open subsets of $\mathbb P^k$ given by
\[
\mathcal U_N^{\kappa}(W, \epsilon) \defeq \{ 
U\subset \mathbb P^k 
\colon 
\exists x \in W
\text{ such that }   
U=U(N,x,\kappa, \epsilon)
\}
\]
where $U(N,x,\kappa, \epsilon)$
 is as in Definition \ref{d:U}.
Recall in particular that each $U(N,x,\kappa, \epsilon)$ (if it exists)
 is an open neighbourhood of $x$.
	Given $\epsilon>0$ and
	$U \in \bigcup_{\kappa>0, N\geq 0} \mathcal U_N^\kappa(\mathbb P^k, \epsilon)$, 
	we denote by $N(U),\kappa (U)$, and
	$z(U)$
	the parameters associated to $U$ as in that definition, i.e., such that
	$$f^{N(U)} (U) = B(f^{N(U)}(z(U)),\kappa(U) \, e^{-N(U)M\epsilon}),$$
	where
$M$ is as in \eqref{eq:def-M}.

\begin{rmk} \label{rmk_3.1}
Let $U_1 \neq U_2$ have the same parameters $N=N(U_i)$ and
	$\kappa=\kappa(U_i)$
	 and assume that $z(U_1)$ and $z(U_2)$ satisfy $f^N (z(U_1))= f^N (z(U_2)) = w$.
	 Then, we necessarily have $U_1 \cap U_2 = \emptyset$, as both $U_1$ and $U_2$
	 correspond to an inverse branch of $f^N$ defined on a subset of
	$B(w,\kappa)$ containing $w$.
\end{rmk}

We fix now $\nu \in \mathcal M^+ (f)$ and let $Z_\nu$ be as in Definition \ref{d:Z}.
We denote as before by $\chi_{\min}>0$ the smallest Lyapunov exponent of $\nu$.
For every $0<\epsilon\ll \chi_{\min}$, we fix 
$Z_\nu(\epsilon)$, $n(\epsilon)$,
and $r(\epsilon)$ as given by Corollary \ref{cor_cordistortion}
and Lemma \ref{cor_cordistortion4}.

By Theorem \ref{thm_11.4.2} and Remark \ref{r:1142}, for $\hat \nu$-almost every $\hat x\in  Z_\nu(\epsilon)$ 
and every $0<\kappa<r(\epsilon)$
there exist positive
integers $m_1(\epsilon, \hat x)\geq n(\epsilon)$ and $m_2 (\epsilon,\kappa)  \geq 1$
such that the conclusion of Theorem \ref{thm_11.4.2} holds for $N\geq m_1(\epsilon, x) + m_2 (\epsilon,\kappa)$. 
For every $m\in\mathbb N$,
consider the set  $Z_\nu(\epsilon, m) \defeq \{ \hat x \in Z_\nu(\epsilon) \colon n(\epsilon) \le m_1(\epsilon, \hat x) \le m\}$.
 Since $\hat \nu (Z_\nu(\epsilon) \setminus  Z_\nu(\epsilon, m))\to 0$ as $m\to \infty$,
for every $0<\epsilon \ll \chi_{\min}$
 there exists $m(\epsilon) \geq n(\epsilon)$
  such that $\nu\left(\pi(Z_\nu(\epsilon))  \setminus \pi(Z_\nu(\epsilon, m(\epsilon)))  \right)<\epsilon$.
  For every $0<\epsilon\ll \chi_{\min}$,
  we define
\begin{equation}\label{d:Zstar}
Z^\star_\nu(\epsilon) \defeq  Z_\nu(\epsilon, m(\epsilon))
\quad 
\mbox{ and }
\quad 
Z^\star_\nu \defeq \cup_{0<\epsilon\ll\chi_{\min}} Z_\nu^\star (\epsilon).
 \end{equation}
By definition, the conclusion of Theorem \ref{thm_11.4.2} holds
for every $x \in \pi(Z_\nu^\star (\epsilon))$, with $m_1(\epsilon, x)$ independent of $x$.
This fact will not be used in this subsection, but will be crucial in the proof of Proposition \ref{prop_2.1}.
Observe also that $\hat \nu (Z_\nu^\star)=1$ and $\nu (\pi(Z_\nu^\star))=1$.

\begin{rmk}\label{rmk:s4-hyperbolic}
As in Remark \ref{rmk:section2-hyperbolic}, when $X$ is uniformly expanding,
for every $\nu \in \mathcal M_X^+(f)$ we can assume that
$Z_\nu=Z^\star_\nu=Z_\nu^\star(\epsilon)$ for all $0<\epsilon\ll\chi_{\min}$, and that
$\pi(Z_\nu)=X$.
\end{rmk}

We first fix $0<\epsilon\ll \chi_{\min}$ and define a quantity $\VD_\nu^\epsilon(Y)$ for 
every subset
 $Y \subseteq \pi(Z^\star_\nu (\epsilon))$. The definition will depend on both $f$ and $\nu$.

\medskip

For $0< \kappa < r(\epsilon)$ and $\mathbb N \ni N^\star \geq n(\epsilon)$,
we denote by $\mathcal U (\epsilon, \kappa, N^\star)$ the collection of open sets
\begin{equation}\label{eq:def-U-ekN}
\mathcal U(\epsilon, \kappa, N^\star) \defeq
 \bigcup_{N \geq N^\star} \mathcal U^\kappa_{N} (\pi(Z^\star_\nu(\epsilon)), \epsilon).
\end{equation}
 
\begin{lem}\label{l:cover} 
For  every $0<\epsilon\ll \chi_{\min}$, $0< \kappa <r(\epsilon)$ and $\mathbb N \ni N^\star \geq n(\epsilon)$,
the collection $\mathcal U(\epsilon,  \kappa, N^\star)$ is an open cover of $\pi(Z^\star_\nu(\epsilon))$.
\end{lem}

\begin{proof}
It follows from Lemma \ref{l:nu-U-compare} and the fact that
$Z^\star_\nu (\epsilon)\subseteq Z_\nu(\epsilon)$ that $U(N,x,\kappa, \epsilon)$
is well-defined (and is an open neighbourhood of $x$) for all $x \in \pi(Z^\star_\nu (\epsilon))$, $0<\kappa<r(\epsilon)$,
and $N\geq n(\epsilon)$. The assertion follows.
\end{proof}

\begin{defn}\label{d:epsilon-cover}
For every 	 $0<\epsilon\ll \chi_{\min}$,
 $Y\subseteq \pi(Z^\star_\nu (\epsilon))$, and $0<\kappa<r(\epsilon)$,
an \emph{$(\epsilon,\kappa)$-cover of $Y$} is a
countable cover $\{U_i\}_{i\geq 1}$
of $Y$ with the property that 
$U_i \in \mathcal U (\epsilon, \kappa, n(\epsilon))$ for all $i$.
An \emph{$\epsilon$-cover} 
is an $(\epsilon,\kappa)$-cover for some $0<\kappa<r(\epsilon)$.
\end{defn}

For every $\alpha \geq 0$ and $Y\subseteq \pi(Z^\star_\nu (\epsilon))$, we define $\Lambda^\epsilon_\alpha (Y)\in [0, + \infty]$ as
\begin{equation}\label{eq:def-lambda}
\Lambda^\epsilon_\alpha (Y)
\defeq
\limsup_{ \kappa \to 0}
\Lambda^{\epsilon, \kappa}_\alpha (Y),
\quad \mbox{where}
\quad
\Lambda^{\epsilon, \kappa}_\alpha (Y)\defeq
\lim_{N^\star \to \infty} \inf_{\{U_i\}}
\sum_{i \ge 1} \volume(U_i)^\alpha
\end{equation}
and the infimum in the second expression is taken over all $(\epsilon,\kappa)$-covers $\{U_i\}_{i \ge 1}$ of $Y$ with $U_i\in \mathcal U(\epsilon, \kappa, N^\star)$ for all $i \ge 1$. Observe that the limit in the second expression above is well-defined, and equal to a supremum over $N^\star \geq n(\epsilon)$,
as the $\mathcal U(\epsilon, \kappa, N^\star)$'s are decreasing collections of covers 
for
 $N^\star \to \infty$. 
We will see below that the function $\alpha \mapsto \Lambda_\alpha^{\epsilon, \kappa}(Y)$ is essentially independent of $\kappa$; see Lemma \ref{lem_indepkappa}.
Hence we will be able to use this approximated 
version of $\Lambda^\epsilon_\alpha (Y)$ in order to study its properties.

\begin{lem}\label{l:non-increasing}
For every $0<\epsilon\ll \chi_{\min}$, $0<\kappa<r(\epsilon)$,
and
 $Y\subseteq \pi(Z^\star_\nu (\epsilon))$,
 the following assertions hold:
 \begin{enumerate}
\item the functions $\alpha \mapsto \Lambda^{\epsilon,\kappa}_\alpha (Y)$  and  $\alpha \mapsto \Lambda^\epsilon_\alpha (Y)$ 
are non-increasing;
\item if 
$ \Lambda^{\epsilon,\kappa}_{\alpha_0} (Y) < \infty$
(resp. $ \Lambda^{\epsilon}_{\alpha_0} (Y) < \infty$)
for some $\alpha_0 \geq 0$,
 then $ \Lambda^{\epsilon,\kappa}_\alpha (Y) = 0$
 (resp. $ \Lambda^{\epsilon}_\alpha (Y) = 0$) 
 for all $\alpha>\alpha_0$.
 \end{enumerate}
\end{lem}
\begin{proof}
By the definition \eqref{eq:def-lambda} of $\Lambda_\alpha^{\epsilon} (Y)$ and $\Lambda_\alpha^{\epsilon, \kappa} (Y)$, it is enough to show the two assertions for $\Lambda^{\epsilon, \kappa}_\alpha (Y)$ for a given $\kappa$ as in the statement.

The first property is clear from the definition of $\Lambda^{\epsilon,\kappa}_\alpha(Y)$ and the fact that, up to taking $N^\star$ sufficiently large, we can assume that the volume of all the $U_i$'s
 is less than 1 in the definition of $\Lambda^{\epsilon,\kappa}_\alpha (Y)$; see Lemma \ref{l:nu-U-compare} (2). If $\alpha_1 < \alpha_2$, then, for every $\eta>0$ and up to taking $N^\star$ sufficiently large, for every $U\in \mathcal U (\epsilon,  \kappa, N^\star)$ we also have
	\[
	\volume(U)^{\alpha_2}
	= \volume(U)^{\alpha_1 + (\alpha_2-\alpha_1)} \leq
	\eta^{(\alpha_2 - \alpha_1)}\cdot
	 \volume(U)^{\alpha_1}.
	\]
As $\eta$ can be taken arbitrarily small and $\Lambda^{\epsilon,\kappa}_{\alpha_1} (Y)$ is finite, this gives $\Lambda^{\epsilon,\kappa}_{\alpha_2} (Y)=0$. The assertion follows.
\end{proof}

Because of Lemma \ref{l:non-increasing}, the following definition is well-posed.

\begin{defn}\label{d:vd-X}
For every $0<\epsilon\ll \chi_{\min}$ 
and
	$Y\subseteq \pi(Z^\star_\nu (\epsilon))$, we set
	\[
	\VD^\epsilon_\nu(Y) \defeq 
	\sup 
	\{ \alpha :  \Lambda^\epsilon_\alpha(Y) = \infty \}=
	\inf\{ \alpha : \Lambda^\epsilon_\alpha(Y) = 0\}.
	\]
	Similarly, for every $0<\kappa<r(\epsilon)$, we also set
	\[\VD^{\epsilon, \kappa}_{\nu}(Y) \defeq
	\sup \{ \alpha :  \Lambda^{\epsilon,\kappa}_\alpha(Y) = \infty\}=
	 \inf\{\alpha : \Lambda^{\epsilon,\kappa}_\alpha (Y)  = 0\}.\]
\end{defn}

\begin{rmk}
The definition of $\VD^{\epsilon,\kappa}_\nu (Y)$
will not be needed in this section, but  Lemma \ref{lem:vdek-vde} will be used 
in the proof of Proposition \ref{prop_confmeasure2}.
\end{rmk}

\begin{lem}\label{l:increasing-epsilon}
For every $0<\epsilon\ll \chi_{\min}$, $0<\kappa<r(\epsilon)$,
 and 
  $Y_1\subseteq Y_2 \subseteq \pi(Z^\star_\nu (\epsilon))$, we have
 \[\VD_\nu^{\epsilon,\kappa} (Y_1)\leq \VD_\nu^{\epsilon,\kappa} (Y_2)
 \quad
 \mbox{ and }
 \quad \VD_\nu^\epsilon (Y_1)\leq \VD_\nu^\epsilon (Y_2).\]
\end{lem}

\begin{proof}
For every $0<\kappa<r(\epsilon)$,
every $(\epsilon,\kappa)$-cover 
of $Y_2$ is also an $(\epsilon,\kappa)$-cover of $Y_1$.
The assertion follows.
\end{proof}

In the next lemma, we will use the following form of Besicovitch's covering
theorem \cite{Besicovitch45}.

\begin{thm} \label{thm_Besicovitch}
Let $n \ge 1$ be an integer. There exists a constant $b(n) > 0$ such that the following claim is true.
If $A$ is a bounded subset of $\mathbb{R}^{n}$, then for any function $r \colon A \to (0,\infty)$ there exists
a countable subset
$\{x_m \colon m \in \mathbb N\}$
of $A$ 
such that the collection of open balls
$\mathcal{B}(A,r) \defeq \{B(x_m, r(x_m))\colon m\in \mathbb N \}$ covers $A$ and can be decomposed into $b(n)$
families whose elements are disjoint.
\end{thm} 

\begin{lem}\label{lem_vdbdd-epsilon}
For every $0<\epsilon \ll \chi_{\min}$, $0<\kappa<r(\epsilon)$,
and
 $Y \subseteq \pi(Z_\nu^\star (\epsilon))$,
 we have
 \[\VD^\epsilon_{\nu}(Y) \le 1 \quad \mbox{ and }
 \quad
 \VD^{\epsilon,\kappa}_{\nu}(Y) \le 1.
 \]
\end{lem}

\begin{proof}
By Lemma \ref{l:increasing-epsilon}, it is enough to show the statement for $Y= \pi (Z^\star_\nu(\epsilon))$.
By \eqref{eq:def-lambda} and the Definition \ref{d:vd-X} of $\VD^\epsilon_\nu (Y)$
and $\VD^{\epsilon,\kappa}_\nu (Y)$, it is enough to show that, for any $0<\kappa<r(\epsilon)$
and $N^\star> n(\epsilon)$, there exists an $(\epsilon,\kappa)$-cover $\{U_i\}$
of $\pi(Z_\nu^\star(\epsilon))$ with $N(U_i)\geq N^\star$ for all $i \ge 1$ and such that
\[\sum_{i}\volume (U_i)<C<\infty\]
for some constant $C$ independent of $\kappa$. Indeed, by \eqref{eq:def-lambda}, this shows that 
$\Lambda_1^{\epsilon,\kappa}(\pi(Z^\star_\nu(\epsilon)))<\infty$,
which implies
 that $\VD_{\nu}^{\epsilon,\kappa}(\pi(Z^\star_\nu(\epsilon))) \le 1$.
 As $C$ is independent of $\kappa$, this also gives 
$\Lambda_1^\epsilon(\pi(Z^\star_\nu(\epsilon)))<\infty$
and thus
$\VD_{\nu}^\epsilon(\pi(Z^\star_\nu(\epsilon))) \le 1$, as desired.

\medskip

Fix $N^\star \geq n(\epsilon)$ and $0<\kappa<r(\epsilon)$.
Theorem \ref{thm_Besicovitch} applied with $n=2k$, $A\defeq f^{N^\star}(\pi(Z^\star_\nu(\epsilon)))$,
and $r \equiv \kappa \, e^{-N^\star M \epsilon}$, gives $b(2k)$
collections $\mathcal B_j$, $1\leq j \leq b(2k)$, of disjoint open balls
$\{B_{j,l}\}_{l \ge 1}$ centred on $f^{N^\star}(\pi(Z^\star_\nu(\epsilon)))$ and of radius $r$ such that  $A\subset \cup_{j,l}B_{j,l}$.
We work here in local charts; see also Remark \ref{rmk_lift}.
 
Consider an element $B_{j^\star, l^\star}$ of the collection $\{B_{j,l}\}_{j,l}$.
By construction, its center belongs to $f^{N^\star}(Z^\star_\nu (\epsilon))$.
Denote by  $x_{j^\star,l^\star}^{1}, \dots, x_{j^\star,l^\star}^{m(j^\star, l^\star)}$
the preimages by $f^{N^\star}$ of the center of $B_{j^\star, l^\star}$
which belong to $\pi(Z^\star_\nu(\epsilon))$.
For each $1\leq q \leq m(j^\star, l^\star)$, choose an orbit $\hat{x} = \hat {x} (j^\star,l^\star,q) \in Z_\nu^\star(\epsilon)$
such that 
$\pi_0 (\hat {x} (j^\star,l^\star, q)) = x_{j^\star,l^\star}^{q}$
(observe that  $\pi_{N^\star} (\hat {x} (j^\star, l^\star,q))$
 is necessarily the center of $B_{j^\star, l^\star}$).
 The inverse branch $f_{T^{N^\star} (\hat {x} (j^\star,l^\star, q))}^{-N^\star}$ is well-defined on the ball $B_{j^\star,l^\star}$
 by the choice of the function $r$ and Lemma \ref{l:nu-U-compare}. More precisely,
 the image of each $B_{j,l}$ under any of such branches is of the
 form $U(N^\star,x,\kappa,\epsilon)$ for some $x \in \pi(Z^\star_\nu(\epsilon))$. 
The 
images associated to the same $B_{j,l}$ are disjoint; see Remark \ref{rmk_3.1}.
Similarly,
 any two such images are also 
 disjoint whenever the corresponding balls $B_{j_0, l_0}$ and $B_{j_1, l_1}$ are disjoint.
 Observe that this in particular applies whenever $j_0=j_1$, since each collection $\mathcal B_{j}$
 consists of disjoint balls.
 
By construction, we have
\[\pi(Z_\nu^\star(\epsilon))
\subseteq
\bigcup_{j =1}^{b(2k)} \bigcup_{l \ge 1} \bigcup_{q=1}^{m(j,l)}
  f_{T^{N^\star} ({\hat x} (j,l,q))}^{-N^\star}(B_{j,l}).
\]
By the arguments above, we have
\[\sum_{j = 1}^{b(2k)}\sum_{l \ge 1} 
\sum_{1\leq q \leq m(j,l)}
\volume \big(
f_{ T^{N^\star}({\hat x} (j,l,q))}^{-N^\star}
(B_{j,l})
\big)
 \leq C' b(2k) \volume(\bbP^k),\]
where the positive constant $C'$ 
is due to the use of local charts,
and is in particular independent of $\kappa$.
This completes the proof.
\end{proof}

Observe that, for every $0<\gamma<1$, there exists a positive 
integer $\theta = \theta_\gamma$ with the property that it is
possible to cover any
ball of radius $r$ in $\mathbb P^k$ with a finite number
$\theta$ of open balls of radius $\gamma r$
(the constant $\theta$ also depends on the dimension $k$, but we omit this dependence since $k$ is fixed).

For every $0<\epsilon \ll \chi_{\min}$, define also the constant
\begin{equation}\label{eq:def-ell}
\ell(\epsilon) \defeq \frac{L_\nu+k(2M+1)\epsilon}{L_\nu - k\epsilon}.
\end{equation}
Observe that $\ell(\epsilon)>1$ for all $\epsilon$ as above, and we have $\ell(\epsilon) = 1 +O(\epsilon)$
for $\epsilon \to 0$.

\begin{lem}\label{lem_indepkappa}
Fix $0 < \epsilon \ll \chi_{\min}$,
$0 < \kappa_1 < \kappa_2 < r(\epsilon)/3$, and $Y\subseteq \pi(Z^\star_\nu (\epsilon))$. 
For every $\alpha \ge 0$, we have
\begin{equation} \label{eq_lambdaineqkappa}
\left(
\kappa_1 / \kappa_2^{\ell(\epsilon)}
\right)^{2k\alpha}
\Lambda_{\alpha \ell(\epsilon)}^{\epsilon, \kappa_2} (Y)
\leq
\Lambda_{\alpha}^{\epsilon, \kappa_1} (Y)
\leq 
\theta_{3\kappa_2/\kappa_1} 
\left(
\kappa_1^{\ell(\epsilon)}/\kappa_2
\right)^{2k\alpha}
\Lambda_{\alpha \ell(\epsilon)^{-1}}^{\epsilon, \kappa_2} (Y),
\end{equation}
where  $\ell(\epsilon)>1$ is as in \eqref{eq:def-ell}
 and $\theta_{3\kappa_2/\kappa_1}>0$ is as above.
In particular, for every $0<\epsilon\ll \chi_{\min}$, $Y\subseteq \pi(Z_\nu^\star(\epsilon))$,
 and $0<\kappa_1<\kappa_2 < r(\epsilon)/3$,
we have
\begin{equation} \label{eq_ineqkappa}
\ell(\epsilon)^{-1}
\VD^{\epsilon, \kappa_2}_\nu (Y)
\leq
\VD^{\epsilon, \kappa_1}_\nu (Y) \leq
\ell(\epsilon) \VD^{\epsilon, \kappa_2}_\nu (Y).
\end{equation}
\end{lem}

\begin{proof}
We first show the inequality
 $(\kappa_1/\kappa_2^{\ell(\epsilon)})^{2k}
\Lambda_{\alpha \ell(\epsilon)}^{\epsilon, \kappa_2} (Y)
\leq
\Lambda_{\alpha}^{\epsilon, \kappa_1} (Y)$.
We can assume that 
$\Lambda_\alpha^{\epsilon, \kappa_1} (Y)<\infty$.
Fix $\eta$ such that 
$\Lambda_\alpha^{\epsilon, \kappa_1} (Y)<\eta <\infty$.
There exists
an $(\epsilon,\kappa_1)$-cover $\{U_i\}_{i\ge1}$ of $Y$, with each $U_i$ 
of the form $U_i = U(N_i,x_i,\kappa_1,\epsilon)$, such that
$$\sum_{i \ge 1} \volume(U_i)^{\alpha} < \eta.$$
For each $i \ge 1$, set $V_i \defeq U(N_i,x_i,\kappa_2,\epsilon)$. Since $\kappa_1 < \kappa_2<r(\epsilon)$, 
we have $U_i \subset V_i$ for all $i$,
and the collection $\{V_i\}_{i \ge 1}$ is an $(\epsilon,\kappa_2)$-cover of $Y$. 
Moreover, by Lemma \ref{l:nu-U-compare} (2) and the definition of $\ell(\epsilon)$,
for every $i$
we have
$$\frac{\volume(V_i)^{\ell(\epsilon)}}{\volume(U_i)} 
\le \frac{\kappa_2^{2k\ell(\epsilon)}}{\kappa_1^{2k}}
 \frac{e^{-2N_i(L_\nu-k\epsilon)\ell(\epsilon)}}{e^{-2N_i(L_\nu+k(2M+1)\epsilon)}} =
\left(
\kappa_2^{\ell(\epsilon)}/\kappa_1
\right)^{2k}.$$
It follows that 
$$\sum_{i \ge 1} \volume(V_i)^{\alpha \ell(\epsilon)} <  
\left(
\kappa_2^{\ell(\epsilon)} / \kappa_1
\right)^{2k\alpha} \eta.$$
By the choice of $\eta$, this shows that $\Lambda^{\epsilon, \kappa_2}_{\alpha\ell(\epsilon)}(Y) \leq (\kappa_2^{\ell(\epsilon)}/\kappa_1)^{2k\alpha} \Lambda_{\alpha}^{\epsilon, \kappa_1} (Y)$, as desired.
By the definition of $\VD_\nu^{\epsilon, \kappa}(Y)$,
 this also shows the first inequality in (\ref{eq_ineqkappa}).

\medskip

We now show the second inequality in \eqref{eq_lambdaineqkappa}. As above, this also shows the second inequality in \eqref{eq_ineqkappa}.
We can assume that $\Lambda_\alpha^{\epsilon, \kappa_2} (Y) <\infty$.
Fix $\eta$ such that $\Lambda_\alpha^{\epsilon, \kappa_2} (Y)<\eta <\infty$.
There exists an $(\epsilon,\kappa_2)$-cover $\{U_i\}_{i \ge 1}$ of $Y$, 
with each $U_i$ of the form $U_i = U(N_i, x_i, \kappa_2, \epsilon)$,
 such that 
$$\sum_{i \ge 1} \volume(U_i)^\alpha < \eta.$$
By Definition \ref{d:U}, for each $i \ge 1$, we have 
$A_i \defeq f^{N_i}(U_i) = B(f^{N_i}(x_i), \kappa_2 \, e ^{-N_iM\epsilon})$. 
By the definition of $\theta_{3\kappa_2/\kappa_1}$,
one can cover any ball $A_i$ with $\theta_{3\kappa_2/\kappa_1}$ 
open balls of radius $(\kappa_1/3) e^{-N_i M\epsilon}$.
In particular, these balls cover $f^{N_i} (Y\cap U_i)\subseteq A_i$.
Up to removing from the collection the balls not intersecting
$f^{N_i} (Y\cap U_i$) and replacing all the other balls 
with balls of the same center and radius $\kappa_1 e^{-N_i M\epsilon}$,
we see that we can cover
$f^{N_i} (Y\cap U_i)$ with $\theta_{3\kappa_2/\kappa_1}$ 
balls of radius $\kappa_1 e^{-N_i M \epsilon}$ and centred at
 points of $f^{N_i} (Y\cap U_i)$.
 For every $i$, we denote by $\{B_{i,j}\}_{1\leq j\leq J_i}$
 (for some $1\leq J_i \leq \theta_{3\kappa_2/\kappa_1}$),
 the collection of the balls of radius $\kappa_1 \, e^{-N_i M \epsilon}$
  constructed above. By construction, for every $i$ we
   have
$$f^{N_i} (Y\cap U_i) 
\subseteq \bigcup_{j=1}^{J_i} B_{i,j}.$$
For every $i,j$, set $V_{i,j}\defeq g_i (B_{i,j})$, where 
$g_i$ is the inverse branch of $f^{N_i}$ defined in a neighbourhood of $f^{N_i} (z_i)$
 that sends
 $f^{N_i} (z_i)$ to $z_i$. 
It follows that, for each $i$,
the collection 
$\{V_{i,j}\}_{1\leq j\leq J_i}$
is an $(\epsilon,\kappa_1)$-cover of $U_i\cap Y$. By Lemma \ref{l:nu-U-compare} (2) and the definition of $\ell(\epsilon)$,
for every $i$ and $1\leq j\leq J_i$
we have
$$\frac{\volume(V_{i,j})^{\ell(\epsilon)}}{\volume(U_i)} \le \frac{\kappa_1^{2k\ell(\epsilon)}}{\kappa_2^{2k}} \frac{e^{-2N_i(L_\nu-k\epsilon)\ell(\epsilon)}}{e^{-2N_i(L_\nu+k(2M+1)\epsilon)}} =
\left(
{\kappa_1^{\ell(\epsilon)}}/{\kappa_2}
\right)^{2k}.$$
Summing over $i$ and $j$, we obtain
\begin{align*}
\sum_{i\ge 1}\sum_{j=1}^{J_i} \volume (V_{i,j})^{\alpha \ell(\epsilon)} &\le \sum_{i \ge 1} \left(\theta_{3\kappa_2/\kappa_1} 
\left (
{\kappa_1^{\ell(\epsilon)}}/{\kappa_2}
\right)^{2k\alpha} \volume(U_i)^\alpha \right)\\
& \le \theta_{3\kappa_2/\kappa_1} \left (
{\kappa_1^{\ell(\epsilon)}}/{\kappa_2}
\right)^{2k\alpha} \eta.
\end{align*}
It follows that $\Lambda^{\epsilon, \kappa_1}_{\alpha\ell(\epsilon)}(Y) \leq \theta_{3\kappa_2/\kappa_1} (\kappa_1^{\ell(\epsilon)}/\kappa_2)^{2k\alpha} \Lambda_{\alpha}^{\epsilon, \kappa_2} (Y)$, as desired.
This completes the proof.
\end{proof}

\begin{lem}\label{lem:vdek-vde}
For every $0<\epsilon \ll \chi_{\min}$, $0<\kappa <r(\epsilon)/3$,
and $Y\subseteq \pi (Z^\star_\nu (\epsilon))$, we have
\[
|\VD^{\epsilon, \kappa}_\nu  (Y)- \VD^\epsilon_\nu (Y)|
\leq 
 (\ell(\epsilon)-1)
\min 
\big\{\VD^{\epsilon}_\nu (Y), \VD^{\epsilon, \kappa}_\nu (Y) \big\}
\leq
 \ell(\epsilon)-1,\]
where $\ell(\epsilon)>1$ is as in \eqref{eq:def-ell}.
\end{lem}

\begin{proof}
Fix $0<\epsilon\ll \chi_{\min}$ and $Y\subseteq \pi (Z^\star_\nu (\epsilon))$.
It follows from Lemma \ref{lem_indepkappa} that there exists $\alpha_0=\alpha_0(\epsilon, Y) \in [0,+\infty]$ such that
\begin{equation}\label{eq:bound-vdek}
\alpha_0  \leq \VD^{\epsilon, \kappa}_\nu (Y)\leq \alpha_0 \ell (\epsilon) \quad \mbox{ for every } 0<\kappa<r(\epsilon).
\end{equation}
Take any $\eta>0$.
It follows from the definition
\eqref{eq:def-lambda}
of $\Lambda_\alpha^{\epsilon}(Y)$ and $\Lambda_{\alpha}^{\epsilon, \kappa} (Y)$
that 
 $\Lambda_{\alpha_0 -\eta}^{\epsilon} (Y)=+\infty$
 (assuming $\alpha_0>0$ and $0<\eta<\alpha_0$)
  and
  $\Lambda_{\alpha_0 \ell(\epsilon)+\eta}^{\epsilon} (Y)=0$
  (assuming $\alpha_0<\infty$). Since $\eta$ is arbitrary,
  we deduce from the definition of $\VD^\epsilon_\nu (Y)$ that
  \begin{equation}\label{eq:bound-vde}
  \alpha_0 \leq \VD^{\epsilon}_\nu (Y) \leq \alpha_0 \ell (\epsilon).
  \end{equation}
  The first inequality in the statement 
  follows from \eqref{eq:bound-vdek}
  and \eqref{eq:bound-vde}. 
  The second one follows from
  Lemma \ref{lem_vdbdd-epsilon}.
\end{proof}

\begin{rmk} \label{rmk_center}
In Definition \ref{d:epsilon-cover}, we do not require $z(U_i)\in Y$ for any of the $U_i$ in an $\epsilon$-cover of $Y$, but we could also 
	define 
	$\VD_{\nu}^{\epsilon}(Y)$
	(and $\VD_{\nu}^{\epsilon, \kappa} (Y)$)
	for $Y \subseteq \pi(Z^\star_\nu (\epsilon))$ by only using 
	sets $U_i$ 
	such that 
	$z(U_i) \in Y$, rather than $z(U_i) \in \pi(Z^\star_\nu (\epsilon))$,
	in the definition of $\Lambda^{\epsilon,\kappa}_\alpha (Y)$.
	Denoting by $\overline{ \Lambda}^\epsilon_\alpha (Y)$ and $\overline{\VD}^\epsilon_\nu(Y)$
	the corresponding quantities, it is straightforward to see that
	$\Lambda^\epsilon_\alpha (Y)\leq \overline{\Lambda}^\epsilon_\alpha (Y)$
	for all $\alpha \geq 0$. Hence, we have 
	$\VD_\nu^\epsilon (Y)\leq \overline{\VD}_\nu^\epsilon(Y)$. 	
	On the other hand, take $\alpha$ such that 
	$\Lambda^\epsilon_\alpha (Y)<\infty$ and
	consider an $(\epsilon, \kappa)$-cover
	 $\{U_i\}_{i\geq 1}$
	of $Y$
	for which the value of $\sum_{i\geq 1} \volume (U_i)^\alpha$ 
	is close to the
	value of $\Lambda^\epsilon_\alpha (Y)$. 
	By the definition of $\Lambda^\epsilon_\alpha (Y)$,
	we can assume that, for all $i$, we have
	$U_i\in  \mathcal U (\epsilon,  \kappa, N_0)$
	for some $ \kappa < r(\epsilon)/3$ and some $N_0\geq n(\epsilon)$.
	Take $i_0$ such that 
	$z(U_{i_0}) \notin Y$. Observe that there must exist 
	$y \in U_{i_0} \cap Y$ and that, since $\kappa< r(\epsilon)/3$, the set
	$U(N(U_{i_0}), y, 3\kappa,\epsilon)$
	is well-defined and contains $U_{i_0}$.
	By similar arguments as in the proof of Lemma \ref{lem_indepkappa},
	this shows that
	$\overline{\Lambda}^{\epsilon,3\kappa}_{\alpha/\ell(\epsilon)}(Y) 
	\leq 3^{2k \alpha/\ell(\epsilon)} \Lambda^{\epsilon,\kappa}_\alpha (Y)$
	for all $0<\kappa<r(\epsilon)/3$,
	which gives $\overline{\VD}_\nu^\epsilon(Y) \le \ell(\epsilon) \VD_\nu^\epsilon (Y)$.
	Similar arguments
	and estimates
	hold for $\VD^{\epsilon, \kappa}_\nu(Y)$.
\end{rmk}

Take now $X\subseteq \pi(Z_\nu^\star)$, and recall that $\nu(\pi (Z_\nu^\star))=1$. For every $0<\epsilon\ll \chi_{\min}$, we set 
\begin{equation}
\label{e:Xepsilon}
X^\epsilon \defeq X\cap \pi(Z^\star_\nu(\epsilon)).
\end{equation}
Observe that,
since 
$Z_\nu^\star =\cup_{0<\epsilon\ll \chi_{\min}} Z^\star_\nu (\epsilon)$, 
we have $\cup_{0<\epsilon\ll \chi_{\min}}X^\epsilon = X$. 

\begin{defn}\label{d:vd}
For every $X\subseteq \pi(Z^\star_\nu)$, the {\it volume dimension} $\VD_\nu (X)$ of $X$ is
	 \[
\VD_\nu (X) \defeq
\limsup_{\epsilon \to 0} \VD_\nu^\epsilon (X^\epsilon).	 
	 \]
\end{defn}

\begin{rmk}
The limsup in Definition \ref{d:vd} is actually a limit; see Section \ref{subsec_eqdefnl},
and in particular Corollary \ref{cor_limit},
and Remark \ref{rmk_5.1}.
\end{rmk}

\begin{lem} \label{lem_monotone}
For every $Y_1 \subseteq Y_2 \subseteq \pi(Z_\nu^\star)$,
 we have 
$\VD_{\nu}(Y_1) \le \VD_{\nu}(Y_2)$. 
\end{lem}
\begin{proof}
By Lemma \ref{l:increasing-epsilon},
for any $0<\epsilon\ll \chi_{\min}$ 
we have
 $\VD_\nu^\epsilon (Y^\epsilon_1) \le \VD_\nu^\epsilon (Y^\epsilon_2)$.
The conclusion follows from Definition \ref{d:vd}. 
\end{proof}

\begin{defn}\label{d:vd-nu}
Take $\nu \in \mathcal{M}^+(f)$.
The {\it volume dimension}  $\VD(\nu)$
of $\nu$ is
$$\VD(\nu) \defeq \inf \big\{\VD_{\nu}(X) : X \subseteq \pi(Z_\nu^\star), \nu(X)=1 \big\}.$$
\end{defn}

\begin{lem}\label{lem_vdbdd}
	For every
	 $\nu \in \mathcal{M}^+(f)$ and
	  $X \subseteq \pi(Z_\nu^\star)$, we have $\VD_{\nu}(X) \le 1$. In particular,
	we have $\VD(\nu) \le 1$ for every
	 $\nu \in \mathcal{M}^+(f)$.
\end{lem}

\begin{proof}
The statement is an immediate consequence of Lemma \ref{lem_vdbdd-epsilon} 
and Definitions \ref{d:vd} and \ref{d:vd-nu}.
\end{proof}

When $X\subseteq \mathbb P^k$ is a uniformly expanding closed invariant set
 for $f$,
 by Remark \ref{rmk:s4-hyperbolic}
we can assume that $\pi(Z^\star_\nu(\epsilon)) = \pi(Z^\star_\nu)=X$  for every
invariant measure $\nu$
on $X$ and every $0<\epsilon\ll \chi_{\min}$.
In particular, the following definition is well-posed
and
defines the term $\VD(J(f))$ in Theorem \ref{thm_main_equalities_hyp}. 
\begin{defn}\label{d:vd-hyp}
	If $X \subseteq \bbP^k$ is uniformly expanding, the {\it volume dimension} of $X$
	is
	$$\VD(X) \defeq \sup_{\nu \in \mathcal{M}_X^+(f)} \VD_{\nu}(X).$$
\end{defn}

We conclude this section with
the next proposition, which in particular shows that, when $k=1$, the volume dimension associated to any $\nu \in \mathcal M^+ (f)$ is equivalent to the Hausdorff dimension.

\begin{prop} \label{rmk_VD}
If $\nu\in\mathcal M^+ (f)$
is such that all the Lyapunov exponents of $\nu$ are equal to $\chi>0$,
 then 
 \begin{enumerate}
 \item $2k\VD_\nu(X)= \HD(X)$ for all $X\subseteq \pi(Z^\star_\nu)$;
 \item
 $2k\VD(\nu) = \HD(\nu)$,
 \end{enumerate}
 where $\HD(X)$ and $\HD(\nu)$ denote the Hausdorff dimension of $X$ and $\nu$, respectively.
\end{prop}

\begin{proof}
Recall that the Hausdorff dimension of $X \subseteq \mathbb P^k$
is defined as
\[
\HD(X)
\defeq \inf \big\{\alpha \colon H_\alpha (X)=0\big\},
\quad \mbox{ where }
H_\alpha (X)\defeq \sup_{\delta > 0} \inf_{\{B_i\}} \sum_{i} (\diam B_i)^\alpha.
\]
The 
infimum in the second expression
is taken over 
all
countable
covers of $X$
by open balls $\{B_i\}$
whose diameter is less than $\delta$. The Hausdorff dimension 
$\HD(\nu)$
of $\nu$ is defined as
\begin{equation}\label{eq:def-HD-nu}
\HD(\nu) \defeq \inf
\big\{\HD(X) \colon X\subseteq \text{Supp } \nu, \, \nu(X)=1\big\}.
\end{equation}

\medskip

In order to prove the first assertion, it is enough to show that
\[
\ell(\epsilon)^{-1} 
\HD (X^\epsilon)\leq
2k\VD_\nu^\epsilon  (X^\epsilon) 
\leq \ell(\epsilon) \HD (X^\epsilon)
\quad \mbox{ for all } X\subseteq \pi (Z_\nu^\star) \text{ and } 0 < \epsilon \ll \chi_{\min},
\]
where we recall that $X^\epsilon$ is defined as in \eqref{e:Xepsilon}
and the constant $\ell(\epsilon)$ is defined in \eqref{eq:def-ell}. Observe that, as $L_\nu = k\chi$, 
we have 
$\ell (\epsilon) = \big(\chi+(2M+1)\epsilon\big)/(\chi-\epsilon)$.

\medskip

We first prove  the inequality $\HD(X^\epsilon) \le 2k \ell(\epsilon) \VD_\nu^\epsilon (X^\epsilon)$. 
Fix  $\alpha_1 > \VD_\nu^\epsilon (X^\epsilon)$.
By Lemma \ref{l:non-increasing}, we have
 $\Lambda^\epsilon_{\alpha_1}(X^\epsilon) = \limsup_{ \kappa \to 0}\Lambda^{\epsilon,\kappa}_{\alpha_1}(X^\epsilon) = 0$.
 Therefore, for any $\eta>0$ and
 up to taking $0<\kappa < r(\epsilon)$ sufficiently small,
 there exists an $(\epsilon,\kappa)$-cover $\{U_i\}_{i \ge 1}$ of $X^\epsilon$
 such that
\begin{equation}\label{eq:sum-HDVD}
\sum_{i \ge 1} \volume(U_i)^{\alpha_1} < \eta.
\end{equation}
Setting $N_i \defeq N (U_i)$,
by Corollary \ref{c:nu-U-compare-equal} (1) and (2)
 we have
\[
\frac{\text{diam}(U_i)^{2k\ell(\epsilon)}}{\volume(U_i)}
\le \frac{2^{2k\ell(\epsilon)}\kappa^{2k\ell(\epsilon)}
e^{-2k N_i
(\chi-\epsilon)
\ell(\epsilon)}}{\kappa^{2k}
e^{-2k N_i(\chi+(2M+1)\epsilon)}}\\
\le 2^{2k\ell(\epsilon)}\kappa^{2k(\ell(\epsilon)-1)}
\leq2^{2k\ell(\epsilon)},
\]
where in the last step
we used the fact that $\kappa < r(\epsilon)<1$.

\medskip

For each $i$, 
define the ball
 $V_i \defeq B(z(U_i), \text{diam}(U_i))$
 of center $z(U_i)$ and radius $\text{diam}(U_i)$.
 Then, $U_i \subseteq V_i$ and $\{V_i\}_{i\geq 1}$ 
 is a cover of $X^\epsilon$ by balls. By the above estimates and \eqref{eq:sum-HDVD},
 we have
\[
\sum_{i \geq 1} \diam(V_i)^{2k\ell(\epsilon)\alpha_1} 
= \sum_{i \geq 1}
(2 \diam(U_i))^{2k\ell(\epsilon)\alpha_1}
\leq 
2^{4k\ell(\epsilon)
\alpha_1}
\sum_{i \geq 1} \volume(U_i)^{\alpha_1} < 
2^{4k\ell(\epsilon)
\alpha_1}\eta.
\]
Therefore, we have $\HD(X^\epsilon) \le 2k\ell(\epsilon)\alpha_1$
and the conclusion follows by taking $\alpha_1 \searrow \VD_\nu^\epsilon (X^\epsilon)$.

\medskip

We now prove the inequality
$2k\ell(\epsilon)^{-1} \VD_\nu^\epsilon (X^\epsilon) \le \HD(X^\epsilon)$.
Fix $\alpha_0$ such that $H_{\alpha_0}(X^\epsilon) =0$.
Then, for any $\eta > 0$, there exists a cover $\{B_i = B(x_i,r_i)\}_{i\geq 1 }$ 
of $X^\epsilon$ consisting of open balls such that 
\begin{equation}\label{eq:hdvd-sum2}
\sum_{i \ge 1} (2r_i)^{\alpha_0} < \eta.
\end{equation}
Fix any $0< \kappa < r(\epsilon)$. By definition of $H_{\alpha_0} (X^\epsilon)$
we can assume that 
\begin{equation}\label{eq:sup-ri}
\sup_i r_i < \kappa \, e^{-n (\epsilon) (\chi + (2M+1) \epsilon)}.
\end{equation}
For each $i \ge 1$, set
\[
N_i \defeq
\big\lfloor 
\frac{\log \kappa-\log r_i}{\chi + (2M+1)\epsilon}\big\rfloor
\]
and $U_i \defeq U(N_i, x_i, \kappa, \epsilon)$.
Observe that $N_i\geq n(\epsilon)$ for all $i$ by \eqref{eq:sup-ri},
hence every $U_i$ is well-defined by Lemma \ref{l:nu-U-compare}.
By Corollary \ref{c:nu-U-compare-equal} (1), for every $i$ we also have
$$B(x_i, r_i) \subseteq B(x_i, \kappa \, e^{-N_i(\chi+(2M+1)\epsilon)}) 
\subseteq 
U_i \subseteq B(x_i, \kappa \, e^{-N_i (\chi - \epsilon)}).$$
In particular, the collection $\{U_i\}_{i\geq 1}$ is an $(\epsilon, \kappa)$-cover of $X^\epsilon$
 and, for all $i$, we also have
\begin{equation}\label{eq:diam-hd}
\volume (U_i)^{\ell(\epsilon)/(2k)} \leq  \kappa^{\ell(\epsilon)}e^{-N_i(\chi+(2M+1)\epsilon)}
 \le  \kappa^{\ell(\epsilon)}e^{\left(1+\frac{\log r_i - \log \kappa}{\chi + (2M+1)\epsilon}\right)(\chi+(2M+1)\epsilon)}
\le
e^{\chi+(2M+1)\epsilon} r_i,
\end{equation}
where we used the facts that $\lfloor x\rfloor \geq x-1$ for every $x>0$
and that $\kappa<r(\epsilon)<1$.

\medskip

It follows from
\eqref{eq:hdvd-sum2}  and
\eqref{eq:diam-hd} that
\[
\sum_{i\geq 1} \volume(U_i)^{\alpha_0 \ell(\epsilon)/(2k)} 
  \leq \sum_{i \ge 1} \left(\frac{1}{2}
  e^{\chi+(2M+1)\epsilon}\right)^{\alpha_0} (2r_i)^{\alpha_0} 
 <  \left(\frac{1}{2}
e^{\chi+(2M+1)\epsilon}\right)^{\alpha_0} \eta.
\]
Therefore, 
for every $0< \kappa<r(\epsilon)$,
we have
\begin{equation}
\label{eq:final-hdvd2}
\Lambda^{\epsilon,\kappa}_{\alpha_0\ell(\epsilon)/(2k)} (X^\epsilon)< \left(\frac{1}{2}
e^{\chi+(2M+1)\epsilon}\right)^{\alpha_0} \eta
<\left(\frac{1}{2}e^{\chi+(2M+1)\epsilon}\right)^{\alpha_0} \eta.
\end{equation}
 Taking
the limsup over $\kappa$ in the left hand side of
\eqref{eq:final-hdvd2},
by 
\eqref{eq:def-lambda} 
we obtain
$\Lambda^{\epsilon}_{\alpha_0\ell(\epsilon)/(2k)} (X^\epsilon) < \infty$.
Therefore,
 we have
  $2k\ell(\epsilon)^{-1} \VD^{\epsilon}_\nu(X^\epsilon) \le \HD(X^\epsilon)$. 
   This completes the proof of the 
  first assertion.

\medskip

 We now prove the second assertion. We first show the inequality
  $\HD(\nu) \le 2k \VD(\nu)$. For every $X \subseteq \pi(Z_\nu^\star)$, we have $2k\VD_\nu(X)= \HD(X)$ by the first assertion. By Definition \ref{d:vd-nu},
  we obtain $2k\VD(\nu) = \inf \{\HD(X) : X \subseteq \pi(Z_\nu^\star), \nu(X)  = 1 \}$. By the definition \eqref{eq:def-HD-nu} of $\HD(\nu)$, this implies that $\HD(\nu) \le 2k\VD(\nu)$.

We now prove the inequality $\HD(\nu) \ge 2k \VD(\nu)$. 
Take $X \subseteq \bbP^k$ with $\nu(X) = 1$. Since $\nu(\pi(Z_\nu^\star))=1$, 
we have
 $\nu(X \cap \pi(Z_\nu^\star))=1$.
 Then, by the first assertion and the monotonicity of the Hausdorff dimension, we have
$2k\VD(X \cap \pi(Z_\nu^\star)) = \HD(X \cap \pi(Z_\nu^\star)) \le \HD(X)$. By the definition \eqref{eq:def-HD-nu} of $\HD(\nu)$, 
we deduce that 
$\inf \{2k\VD(X \cap \pi(Z_\nu^\star)) : \nu(X) 
= 1 \} \le \HD(\nu)$. By Definition \ref{d:vd-nu},
this gives $2k\VD(\nu) \le \HD(\nu)$
and completes the proof.
\end{proof}

\subsection{An equivalent definition of $\VD(\nu)$}
\label{subsec_eqdefnl}

We present here an equivalent definition of the volume dimension for sets $X \subset \pi(Z_\nu^\star)$. This definition in particular allows us to prove that  the $\limsup_{\epsilon\to 0}$ in Definition \ref{d:vd} is always a limit; 
 see also Remark \ref{rmk_5.1} for sets $X \subseteq \pi(Z_\nu^\star)$ with $\nu(X) > 0$.
 The advantage of this definition is that we will not have the small exponential terms $e^{-NM\epsilon}$ in the definition of the sets of the covers. In particular, we work with sets which are more similar to actual Bowen balls of fixed radius. On the other hand, the collection of neighbourhoods associated to any $x$ will in some sense depend on $x$. This section is not necessary in order to obtain the main results of the paper.

\medskip

For every
$0<\epsilon\ll \chi_{\min}$, $0 < \kappa < r(\epsilon)$, $l \in \mathbb N$,
and $W\subseteq \pi(Z_\nu^\star (\epsilon))$, we consider the collection
$\widetilde{\mathcal U}^{\kappa}_l (W, \epsilon)$ of open subsets of $\mathbb P^k$ given by
\[
\widetilde{\mathcal U}_l^{\kappa}(W, \epsilon) \defeq 
\big\{ 
\widetilde{U} \subset \mathbb P^k 
\colon 
\exists x \in W,
\text{ such that }   
\widetilde{U}=\widetilde{U}(n_l,x,\kappa) \big\}.
\]
Here,
 $\{n_l\}_{l\geq 0}$ is the sequence associated to $x\in \pi (Z_\nu(\epsilon))$ by Lemma \ref{cor_cordistortion4}, and, letting $\hat x$ be any element of $Z_\nu (\epsilon)$ with $x_0=x$, we set 
\[
\widetilde U (n_l, x, \kappa)\defeq f_{T^{n_l} (\hat x)}^{-n_l} 
\big(B(f^{n_l}(x_0), \kappa)\big),
\]
where the right hand side of the above expression is well-defined by Corollary \ref{cor_cordistortion}.
For every  $\epsilon$ and $\kappa$ as above and
$\mathbb N \ni N^\star \geq n(\epsilon)$, 
we denote by $\widetilde{\mathcal U} (\epsilon, \kappa, N^\star)$ the collection of open sets
\[
\widetilde{\mathcal U}(\epsilon, \kappa, N^\star) \defeq
\bigcup_{n_l \geq N^\star} \widetilde{\mathcal U}^\kappa_{l} (\pi(Z^\star_\nu(\epsilon)), \epsilon).
\]

For every $\alpha \geq 0$ and $Y\subseteq \pi(Z^\star_\nu (\epsilon))$, we define
$\widetilde{\Lambda}^\epsilon_\alpha (Y)\in [0, + \infty]$
as
\begin{equation}\label{eq:def-lambda-tilde}
	\widetilde{\Lambda}^\epsilon_\alpha (Y)
	\defeq
	\limsup_{ \kappa \to 0} \widetilde{\Lambda}^{\epsilon,\kappa}_\alpha (Y),
	\quad  \text{ where }
	\quad  \widetilde{\Lambda}^{\epsilon,\kappa}_\alpha (Y) \defeq
	\lim_{N^\star \to \infty} \inf_{\{U_i\}}
	\sum_{i \ge 1} \volume(\widetilde{U}_i)^\alpha
\end{equation}
and the infimum is taken over all 
covers $\{\widetilde{U}_i\}_{i \ge 1}$ of $Y$
with $\widetilde{U}_i \in \widetilde{\mathcal U}(\epsilon, \kappa, N^\star)$ for all $i \ge 1$.
As in Lemma \ref{l:non-increasing}, one can show that, for every $0<\epsilon\ll \chi_{\min}$
and $Y\subseteq \pi(Z^\star_\nu (\epsilon))$, the function $\alpha \mapsto \widetilde{\Lambda}^\epsilon_\alpha (Y)$ is non-increasing and that, if $ \widetilde{\Lambda}^{\epsilon}_{\alpha_0} (Y) < \infty$ for some $\alpha_0 \geq 0$, then $\widetilde{\Lambda}^{\epsilon}_\alpha (Y) = 0$ for all $\alpha>\alpha_0$.
As a consequence, the following definition is well-posed.
\begin{defn}\label{d:vd-X-tilde}
	For every $0<\epsilon\ll \chi_{\min}$ and $Y \subseteq Z^\star_\nu (\epsilon)$, we set
	\[
	\widetilde{\VD}^\epsilon_\nu(Y) \defeq 
	\sup \{ \alpha :  \widetilde{\Lambda}^\epsilon_\alpha(Y) = \infty\}=
	\inf\{ \alpha : \widetilde{\Lambda}^\epsilon_\alpha(Y) = 0\}.
	\]
\end{defn}

\begin{lem}\label{l:stab-epsilon-tilde}
For every
 $0<\epsilon_1 < \epsilon_2 \ll \chi_{\min}$ 
	and 
	$Y\subseteq \pi(Z^\star_\nu (\epsilon))$,
	 we have $\widetilde{\VD}^{\epsilon_1}_\nu(Y) = \widetilde{\VD}^{\epsilon_2}_\nu(Y)$.
\end{lem}

\begin{proof}
	The statement is clear since the sets $\widetilde U (n_l, x, \kappa)$ do not depend on $\epsilon$.
\end{proof}

\begin{lem}\label{l:equivalence-vd-tilde}
	For every $0<\epsilon \ll \chi_{\min}$ 
	and 
	$Y\subseteq \pi(Z^\star_\nu (\epsilon))$
	 we have $$\ell(\epsilon)^{-1}\widetilde{\VD}^{\epsilon}_\nu(Y) \le \VD^{\epsilon}_\nu(Y) \le \beta(\epsilon)\widetilde{\VD}^{\epsilon}_\nu(Y),$$ where
	$\ell(\epsilon) > 1$ is as in \eqref{eq:def-ell} and
	$\beta(\epsilon) \defeq \frac{L_\nu+k\epsilon}{L_\nu - k\epsilon} 
	\cdot
	\big( \min_j   \frac{\chi_j - \epsilon}{\chi_j + (2M+1)\epsilon}  - \frac{1}{n(\epsilon)}\big)^{-1}$.
\end{lem}

Observe that $\beta(\epsilon)$ as in the statement above satisfies $\beta(\epsilon)>1$ for all $0<\epsilon \ll \chi_{\min}$
and $\beta(\epsilon)=1+O(\epsilon)$ as $\epsilon\to 0$.

\begin{proof}
We first prove the first inequality. Suppose
$\alpha\geq 0$ is such that
 $\Lambda^\epsilon_\alpha(Y) = 0$. Then for any $\eta > 0$ and $0<\kappa<r(\epsilon)$, there exists a $(\kappa,\epsilon)$-cover $\{U_i\}_{i \ge 1}$ of $Y$ of the form $U_i = U(N_i,x_i,\kappa,\epsilon)$, with $N_i\geq  N^\star \geq n(\epsilon)$ for all $i$, such that $\sum_{i \ge 1} \volume(U_i)^{\alpha} < \eta$. 
	
For each $i \ge 1$, let $l(i)$ be 
such that
 $n_{l(i)}\leq N_i < n_{l(i)+1}$. Such $l(i)$ exists since, by Lemma \ref{cor_cordistortion4}, we have $n_0(x)\leq n(\epsilon)$ for all $x\in \pi(Z_\nu (\epsilon))$. 
For every $i$, we then have $U_i 
	\subset\widetilde U_i \defeq \widetilde{U}(n_{l(i)+1},x_i,\kappa)$,
	see also \eqref{e:U-n-nl+1} in the proof of Lemma \ref{l:nu-U-compare}. 
	In particular, we have $\{\widetilde U_i\}\subset \widetilde{\mathcal U} (\epsilon, \kappa, N^\star)$
	and the sets $\{\widetilde U_i\}$ form a cover of $Y$.
	It follows  from the definition of $\ell(\epsilon)$ that for all $i\geq 1$, we have
		\[
		\frac{\volume(\widetilde{U}(n_{l(i)+1},x_i,\kappa))^{\ell(\epsilon)}}{\volume(U(N_i,x_i,\kappa,\epsilon))} 
		\le \frac{\kappa^{2k\ell(\epsilon)}e^{-2n_{l(i)+1}(L_\nu-k\epsilon)\ell(\epsilon)}}{\kappa^{2k}e^{-2 N_i(L_\nu+k(2M+1)\epsilon)}}
		 = \kappa^{2k(\ell(\epsilon)-1)}e^{2(L_\nu+\kappa\epsilon(2M+1))(N_i-n_{l(i)+1})}
		 \leq 1,
		 \]
where in the last inequality we used the facts the $N_i < n_{l(i)+1}$ and 
$\kappa<r(\epsilon)<1$.
 In particular, we have
\[
\sum_{i \ge 1} \volume\big(\widetilde{U}(n_{l(i)+1},x_i,\kappa)\big)^{\alpha\ell(\epsilon)}
 \le
\sum_{i \ge 1} \volume(U_i)^{\alpha} < 
\eta,
\]
which gives the inequality
$\widetilde{\Lambda}^{\epsilon,\kappa}_{\alpha\ell(\epsilon)} (Y)
\le
 \Lambda^{\epsilon,\kappa}_{\alpha} (Y)$ for any $0<\kappa<r(\epsilon)$.
 Taking 
the
limsup over $\kappa$ as in the definition of $\widetilde{\Lambda}^{\epsilon}_{\alpha\ell(\epsilon)} (Y)$,
we obtain $\widetilde{\Lambda}^{\epsilon}_{\alpha\ell(\epsilon)} (Y) < \infty$.
By the choice of $\alpha$, we deduce
the desired 
 inequality
$\widetilde{\VD}^\epsilon_\nu(Y) \le \ell(\epsilon)\VD^\epsilon_\nu(Y)$.
	
\medskip
	
We now prove the second inequality. Suppose $\widetilde{\Lambda}^\epsilon_\alpha(Y) = 0$.
Then for any $\eta > 0$ and $0<\kappa<r(\epsilon)$, there exists a cover
$\{\widetilde{U}_i\}_{i \ge 1}\subset \widetilde {\mathcal U} (\epsilon, \kappa, N^\star)$ of $Y$,
which each $\widetilde U_i$ of the form $\widetilde{U}_i = \widetilde{U}(n_{l(i)},x_i,\kappa)$,
 such that $\sum_{i \ge 1} \volume(\widetilde{U}_i)^{\alpha} < \eta$. For each $i \ge 1$, set 
$$N_i \defeq  \big\lfloor n_{l(i)} \cdot \min_j  \frac{\chi_j - \epsilon}{\chi_j + (2M+1)\epsilon} \big\rfloor.$$
	
From the definition of $N_i$ and Lemma \ref{l:nu-U-compare},
for all $i\geq 1$, we have 
	$$\widetilde{U}(n_{l(i)},x_i,\kappa) \subset U(N_i, x_i, \kappa, \epsilon).$$
	It follows that, for every $i\geq 1$, we have
	\[\frac{\volume(U(N_i,x_i,\kappa,\epsilon))^{\beta(\epsilon)}}{\volume(\widetilde{U}(n_{l(i)},x_i,\kappa))}
	\le
	\frac{\kappa^{2k\beta(\epsilon)}e^{-2N_i(L_\nu-k\epsilon)\beta(\epsilon)}}{\kappa^{2k}e^{-2n_{l(i)}(L_\nu+k\epsilon)}}
\leq	1,
	\]
	where in the last step we used the facts
that $\kappa<r(\epsilon)<1$ and 	
	 that, 
	since $n_{l(i)}\geq n(\epsilon)$ for all $i\geq 1$ and $\lfloor r\rfloor \geq r-1$ for all $r >0$,
	 we have
	\[
	N_i \frac{L_\nu - k\epsilon}{L_\nu + k\epsilon}
	\beta(\epsilon)
	\geq
	\frac{n_{l(i)}  \cdot \min_j  \frac{\chi_j - \epsilon}{\chi_j + (2M+1)\epsilon } -1}{\min_j  \frac{\chi_j - \epsilon}{\chi_j + (2M+1)\epsilon} - \frac{1}{n(\epsilon)}}
	\geq n_{l(i)} .
	\]
Therefore, we have
$$\sum_{i \ge 1} \volume(U(N_i,x_i,\kappa,\epsilon))^{\alpha\beta(\epsilon)} \le 
\eta,$$
which gives the inequality
$\Lambda^{\epsilon,\kappa}_{\alpha\beta(\epsilon)} (Y)\le
\widetilde{\Lambda}^{\epsilon,\kappa}_{\alpha} (Y)$ for any $0<\kappa<r(\epsilon)$. Taking
the limsup over $\kappa$ as in the definition of $\Lambda^{\epsilon}_{\alpha\beta(\epsilon)}(Y)$, we obtain $\Lambda^{\epsilon}_{\alpha\beta(\epsilon)} (Y) = 0$.
By the choice of $\alpha$, we have
 $\VD^\epsilon_\nu(Y) \le \beta(\epsilon)\widetilde{\VD}^\epsilon_\nu(Y)$. The proof is complete.
\end{proof}

Thanks to Lemma \ref{l:equivalence-vd-tilde}, one can see that the $\limsup_{\epsilon\to 0}$
in the Definition \ref{d:vd-X} is actually a limit.
Recall that, for every $X\subseteq \pi(Z_\nu^\star)$, we denote $X^\epsilon\defeq X\cap \pi(Z^\star_\nu(\epsilon))$.

\begin{cor} \label{cor_limit}
	For every $X \subseteq \pi(Z^\star_\nu)$, we have
	\[\VD_\nu (X)=\lim_{\epsilon \to 0} \VD^\epsilon_\nu (X^\epsilon).\]
\end{cor}

\begin{proof}
	For every $X \subseteq \pi(Z_\nu^\star)$, set
	$\widetilde \VD_\nu (X)\defeq\lim_{\epsilon\to 0} \widetilde \VD_\nu^\epsilon (X^\epsilon)$.
	The limit is well-defined and equal to the supremum over $0<\epsilon \ll \chi_{\min}$ since, for every
	$0<\epsilon_1<\epsilon_2\ll \chi_{\min}$, we have
	$\widetilde{\VD}^{\epsilon_1}_\nu (X^{\epsilon_2}) = \widetilde{\VD}^{\epsilon_2}_\nu (X^{\epsilon_2})$ by Lemma \ref{l:stab-epsilon-tilde}
	and
	$\widetilde{\VD}^{\epsilon_1}_\nu (X^{\epsilon_1}) \geq  \widetilde{\VD}^{\epsilon_1}_\nu (X^{\epsilon_2})$ since $X^{\epsilon_1}\supseteq X^{\epsilon_2}$.
	It follows from Lemma \ref{l:equivalence-vd-tilde} that $\lim_{\epsilon\to 0}\VD^\epsilon_\nu (X^\epsilon)$
	is well-defined and equal to $\widetilde \VD_\nu (X)$. The assertion follows.
\end{proof}

\begin{rmk}
A further possible (equivalent) way to define the volume dimension is the following. Take $\nu\in \mathcal M^+ (f)$. Fix $0<\epsilon\ll \chi_{\min}$ and take $Y\subseteq \pi(Z^\star_\nu (\epsilon))$.
For every $\alpha\geq 0$, define
\[
	\check \Lambda_\alpha^\epsilon (Y) 
	\defeq
	\limsup_{\kappa \to 0}  \check{\Lambda}^{\epsilon,\kappa}_\alpha (Y),
	\quad
	 \text{ where }
	 \quad 
	  \check{\Lambda}^{\epsilon,\kappa}_\alpha (Y) \defeq
	 \lim_{N^\star \to \infty} \inf_{\{U_i\}} 
	\sum_{i \ge 1} e^{-2N(U_i)L_{\nu}(f)\alpha}  \kappa^{2k\alpha},
\]
and the infimum is taken over all countable covers $\{U_i\}_{i \ge 1}\subset \mathcal{U}(\epsilon, \kappa, N^\star)$ of $Y$.
Recall that $L_\nu(f)$ denotes the sum of the Lyapunov exponents of $\nu$.
As in Lemma \ref{l:non-increasing}, one can prove that, for every $0<\epsilon\ll \chi_{\min}$
and $Y\subseteq \pi(Z^\star_\nu(\epsilon))$, the function $\alpha\mapsto  \check \Lambda_\alpha^\epsilon (Y)$
is non-increasing in $\alpha$, and that if $\check \Lambda_{\alpha_0}^\epsilon (Y)$ is finite,
then $\check \Lambda_{\alpha}^\epsilon(Y)=0$ for all $\alpha>\alpha_0$.
Hence, the quantity
	\[
	\check {\VD}_\nu^\epsilon (Y)
	\defeq 
	\inf \big\{\alpha \colon \check \Lambda_{\alpha}^\epsilon(Y) =0 \big\}
	=\sup \big\{ \alpha \colon\check \Lambda_{\alpha}^\epsilon(Y)  =\infty\big\}
	\] 
is well-defined for all $0<\epsilon\ll \chi_{\min}$ and $Y\subseteq \pi(Z^\star_\nu(\epsilon))$.
For every $0<\epsilon\ll\chi_{\min}$, define the constants
	\[
	\ell_-(\epsilon) \defeq \frac{
		L_\nu (f)- k\epsilon}{L_\nu (f)}
	\quad
	\mbox{ and }
	\quad
	\ell_+(\epsilon) \defeq
	\frac{ L_\nu (f)+ (2M+1)k\epsilon}{L_\nu (f)}.
	\]
	Observe that, for all $0<\epsilon\ll\chi_{\min}$, we have
	$\ell_-(\epsilon)<1<\ell_+(\epsilon)$ and 
	$\ell_-(\epsilon), \ell_+ (\epsilon)\to 1$ as $\epsilon \to 0$.
One can show in this case that, for every $0<\epsilon\ll\chi_{\min}$ and  $Y\subseteq \pi(Z^\star_\nu (\epsilon))$, we have
\begin{equation}\label{e:equivalence-2}
\ell_-(\epsilon) \VD_{\nu}^\epsilon(Y) \le \check {\VD}^\epsilon_\nu (Y) \le \ell_+(\epsilon) \VD_{\nu}^\epsilon(Y).
\end{equation}
	
Take now $X\subseteq \pi(Z_\nu^\star)$.
As before, setting $X^\epsilon \defeq X\cap \pi(Z^\star_\nu(\epsilon))$ for every $0<\epsilon\ll \chi_{\min}$, we can define
\[\begin{aligned}
		\check{\VD}_\nu(X) &\defeq \limsup_{\epsilon \to 0}  \VD_\nu^{\epsilon}(X^\epsilon)
		\quad  \mbox{ and }\\
		\check{\VD} (\nu) &  \defeq 
		\inf \big\{
		\check{\VD}_\nu (X) \colon X \subseteq
		\pi(Z_\nu^\star)
		\mbox{ and } \nu(X)=1\big\}.
\end{aligned}
\]
It follows from \eqref{e:equivalence-2}, applied with $Y=X^\epsilon$,
that $\check {\VD}_\nu (X) = \VD_\nu (X)$ for all $X\subseteq \pi(Z_\nu^\star)$,
and that $\check{\VD} (\nu) = \VD (\nu)$.
\end{rmk}

\subsection{From local volume dimensions to volume dimensions} \label{subsec_keyprop}
Fix a measure $\nu \in \mathcal{M}^+(f)$ and $0<\epsilon\ll \chi_{\min}$.
For $x \in \pi(Z^\star_\nu(\epsilon))$, $0<\kappa<r(\epsilon)$, and $N\geq n(\epsilon)$
recall that $\delta_{x} (\epsilon, \kappa, N)$ is defined in \eqref{eq:def-delta-xekN}
and well-defined by Lemma \ref{l:nu-U-compare}. The integer $m_1(\epsilon, x)$
in Theorem \ref{thm_11.4.2}
is uniformly bounded
from above
 for all $x\in \pi(Z^\star_\nu (\epsilon))$ by the definition \eqref{d:Zstar}
of $Z^\star_\nu (\epsilon)$. 
This fact is crucial in the proof of the next statement.
 Recall that a measure is {\it non-atomic} if it does not assign mass to points.
 
\begin{prop}
\label{prop_2.1}
Let $f$ be a holomorphic endomorphism of $\mathbb P^k$
of algebraic degree $d\geq 2$ 
 and take $\nu \in \mathcal{M}^+(f)$.  Assume that $\nu$ is non-atomic.
Fix $\alpha_1, \alpha_2 \geq 0$ and $0<\epsilon\ll \chi_{\min}$.
Let $Y \subseteq \pi(Z^\star_\nu(\epsilon))$ be 
such that $\nu(Y) > 0$. Suppose
that for every 
$0<\kappa < r(\epsilon)$ there exists $m=m(\epsilon,\kappa)\geq n(\epsilon)$ 
such that 
\begin{equation}
 \label{e:assumption-young}
	\alpha_1 
	\le
\delta_x (\epsilon,\kappa, N) 	
	\le 
	\alpha_2
	\quad
	\mbox{ for all } 
x \in Y 	
	\mbox{ and } 
 N\geq m(\epsilon,\kappa).
\end{equation}
Then,  we have
$$ \alpha_1 \le \VD^{\epsilon}_\nu(Y) \le \alpha_2.$$
\end{prop}

\begin{proof}
The proof of the proposition essentially follows the arguments in \cite[Proposition 2.1]{Young}.
Recall that,
for every $0<\kappa<r(\epsilon)$ and $N^\star\geq 0$,
the collection
$\mathcal U (\epsilon, \kappa, N^\star)$
is defined in \eqref{eq:def-U-ekN}
and
is a cover of $Y$; see Lemma \ref{l:cover}.
Define the quantity
$$\alpha (Y, \nu) \defeq
 \inf \Big\{\alpha : 
	\lim_{N^\star \to \infty}
	 \inf_{\mathcal U'} \sum_{ U\in \mathcal U'} \nu(U)^{\alpha} =0 \Big\},$$
 where the infimum is taken over all the sub-covers
  $\mathcal U' \subset \mathcal U( \epsilon, \kappa, N^\star)$ 
 of $Y$. 
Since
$\nu$ is non-atomic and
 $\nu(Y) > 0$, we have $\alpha(Y,
\nu) = 1$; see \cite[Section 14]{Billingsley}. We use here
 the fact that, for every fixed $0<\kappa<r(\epsilon)$, the sets $U(N,x ,\kappa, \epsilon)$ shrink
to $\{x\}$ as $N\to \infty$; see Lemma \ref{l:nu-U-compare}.

\medskip

Fix $\alpha > 1$,  $0<\kappa<r(\epsilon)$, $N^\star \geq m(\epsilon,\kappa)$, and $\eta>0$. Since $\alpha (Y, \nu)=1$, there exists a cover $\mathcal U_0 \subset  \mathcal U(\epsilon, \kappa, N^\star)$ of $Y$ such that 
$$\sum_{U \in \mathcal U_0} \nu(U)^{\alpha} < \eta.$$
By the assumption (\ref{e:assumption-young}) and the choice $N^\star \geq m (\epsilon,\kappa)$,
for every $U \in
	\mathcal U (\epsilon, \kappa, N^\star)$ we have $\volume(U)^{\alpha_2} \leq \nu(U)$.
Hence, we have
\[\sum_{U \in \mathcal U_0} \volume (U)^{\alpha_2 \alpha}\leq \sum_{U \in \mathcal U_0} \nu(U)^{\alpha} < \eta.\]
This shows the inequality
$\Lambda^{\epsilon,\kappa}_{\alpha_2\alpha} (Y) < \eta$ for any $0 < \kappa < r(\epsilon)$. Therefore, we have
 $\limsup_{ \kappa \to 0}\Lambda^{\epsilon,\kappa}_{\alpha_2\alpha} (Y) = \Lambda^{\epsilon}_{\alpha_2\alpha} (Y) < \eta$. Taking $\alpha \searrow 1$, we obtain the inequality
  $\VD_\nu^{\epsilon} (Y)\leq \alpha_2$. 

\medskip

For the other inequality, again by the assumption \eqref{e:assumption-young}, for all $0<\kappa<r(\epsilon)$
and $N^\star \geq m(\epsilon,\kappa)$
 we have $\volume(U)^{\alpha_1} \ge \nu(U)$ for every $U \in \mathcal
U(\epsilon, \kappa, N^\star)$.
Hence, for any cover $\mathcal U_0 \subset \mathcal U 
(\epsilon, \kappa, N^\star)$, 
we have
	\[
	\sum_{U\in \mathcal U_0 } \volume(U)^{\alpha_1}\geq \sum_{U \in \mathcal U_0} \nu(U)\geq \nu(Y).\]
Therefore, we have
 $\Lambda^{\epsilon,\kappa}_{\alpha_1}(Y) \ge \nu(Y)>0$ for any $0 < \kappa < r(\epsilon)$, which gives
  \[\Lambda^{\epsilon}_{\alpha_1}(Y) = \limsup_{ \kappa \to 0}\Lambda^{\epsilon,\kappa}_{\alpha_1}(Y)  \ge \nu(Y)>0.\]
  Hence, we have $\VD_\nu^\epsilon(Y) \ge \alpha_1$. The proof is complete.
\end{proof}

\begin{rmk}\label{r:young}
Let $\nu \in \mathcal{M}^+(f)$ be non-atomic.
Take $X \subseteq \pi(Z_\nu^\star)$ with $\nu (X)>0$. Setting $X^\epsilon \defeq X \cap \pi (Z^\star_\nu(\epsilon))$, assume that for every $0<\epsilon\ll \chi_{\min}$ and $0<\kappa<r(\epsilon)$ there exists $m=m(\epsilon, \kappa)$
and $\alpha_1^0, \alpha_2^0\in \mathbb R$
 such that
 \begin{equation}
\alpha_1  (\epsilon)\le \delta_x (\epsilon,\kappa, N) \le \alpha_2 (\epsilon)
	\quad
	\mbox{ for all } 
x \in X^\epsilon
	\mbox{ and } 
 N\geq m(\epsilon,\kappa)
\end{equation}
for some functions $\alpha_1 (\epsilon) = \alpha_1^0+ O(\epsilon)$ and
$\alpha_2 (\epsilon) = \alpha_2^0 + O(\epsilon)$.
Applying Proposition \ref{prop_2.1} to $X^\epsilon$ instead of $Y$ we see that, for every $0<\epsilon\ll \chi_{\min}$, we have $
\alpha_1 (\epsilon) \leq \VD_\nu^\epsilon (X^\epsilon) \leq \alpha_2 (\epsilon)
$,
which gives 
$\alpha_1^0 \leq \VD_\nu (X) \leq \alpha_2^0$.
\end{rmk}

\section{Proofs of Theorems \ref{thm_main}, \ref{thm_main_equalities_general}, and \ref{thm_main_equalities_hyp}} \label{sec_proof1.3}
In this section, 
$f\colon\mathbb P^k\to \mathbb P^k$ is a holomorphic endomorphism of algebraic degree $d\geq 2$.

\subsection{Proof of Theorem \ref{thm_main}} \label{ss:1.1}
For $\nu \in \mathcal M^+(f)$, recall that $Z_\nu$ is defined in Definition \ref{d:Z} and for every $0<\epsilon\ll \chi_{\min}$ (where $\chi_{\min}>0$ is the smallest Lyapunov exponent of $\nu$), the set $Z^\star_\nu (\epsilon)$ is defined in \eqref{d:Zstar}.

\medskip

Assume first that $\nu$ is atomic.
Since $\nu$ is ergodic, it gives mass only to a finite number of points,. hence it
 satisfies 
$h_\nu(f) = 0$. 
It also follows from the Definition 
\ref{d:vd-nu} 
of $\VD (\nu)$ that 
 $\VD(\nu) = 0$, since the support $S_\nu$ of $\nu$ satisfies 
 $\nu(S_\nu)=1$ and $\VD_\nu (S_\nu)=0$, being finite.
  Therefore, the conclusion follows in this case.
 
  \medskip

We can then assume that  $\nu$ is non-atomic. By Theorem \ref{thm_11.4.2} and Proposition \ref{prop_2.1}, for every $0<\epsilon\ll \chi_{\min}$
	we have 
	$$
	\VD^{\epsilon}_\nu(\pi(Z^\star_\nu(\epsilon)))\leq
	 \frac{h_{\nu}(f)}{2L_{\nu}(f)} + c\epsilon,
$$
	where
	 the constant $c$ is independent of $\epsilon$.
	By Definition \ref{d:vd}, taking $\epsilon \searrow 0$, we obtain the inequality 
	$\VD_\nu (\pi(Z_\nu)) \le (2L_{\nu}(f))^{-1} h_{\nu}(f)$. 
	As $\nu (\pi(Z_\nu))=1$,
	by Definition  \ref{d:vd-nu}, we deduce the inequality
	$\VD(\nu) \le (2L_{\nu}(f))^{-1}  h_{\nu}(f)$. 

\medskip

	In order to prove
	the reversed inequality, let
	$Y_0 \subseteq \pi(Z^\star_\nu)$ be such that
	$\nu(Y_0) = 1$.
	For any $0<\epsilon\ll \chi_{\min}$, 
	applying Proposition 
	\ref{prop_2.1} to $Y_0 \cap \pi(Z^\star_\nu(\epsilon))$, 
	we deduce from Theorem \ref{thm_11.4.2}  that
	\[\VD^{\epsilon}_\nu(Y_0 \cap \pi(Z^\star_\nu(\epsilon))) \geq 
	 \frac{h_{\nu}(f)}{ 2L_{\nu}(f)}
	- c \epsilon,
	\]
	where again the constant $c$ is independent of $\epsilon$; see also
	Remark \ref{r:young}.
	By
	Definition \ref{d:vd},
	we have the inequality
	$\VD_\nu(Y_0) \ge (2 L_{\nu}(f))^{-1}h_{\nu}(f)$.
	As $Y_0$ is arbitrary, 
	it follows from Definition \ref{d:vd-nu} 
	that
	$\VD(\nu) 
	\ge (2L_{\nu}(f))^{-1}h_{\nu}(f).$
	The proof of Theorem \ref{thm_main} is complete.

\begin{rmk} \label{rmk_5.1}
Let $\nu \in \mathcal{M}^+(f)$ be non-atomic and take $X \subseteq \pi(Z_\nu^\star)$
 with $\nu(X) > 0$.
 By Remark \ref{r:young} and with similar arguments as in the proof of
  Theorem \ref{thm_11.4.2},
  it follows that the $\limsup_{\epsilon \to 0}$ 
in Definition \ref{d:vd} is actually a limit.
\end{rmk}

\subsection{Proof of Theorem \ref{thm_main_equalities_general}}\label{ss:1.2}

Let $X \subseteq \bbP^k$ be a closed $f$-invariant set. Define
\begin{equation}\label{e:def-dd-hyd}
		\DD^+_X (f)  \defeq \sup  \big\{ \VD (\nu) \colon \nu \in  \mathcal{M}_X^+ (f) \big\}
\end{equation}
and recall that $\delta_X(f)$, $P^+_X(t)$,
 and  $p^+_X(f)$ are defined in Sections \ref{ss:volume-conformal} and
 \ref{ss:pressure-exp}.
Theorem \ref{thm_main_equalities_general} follows from the following proposition applied with $X=J(f)$.

\begin{prop}\label{prop_zero}
We have $p_X^+(f) = 
2\DD_X^+(f)$.  In particular, the set
$\big\{t\colon  P_X^+(t)=0\big\}$ 
is non-empty.
\end{prop}

\begin{proof}
We first prove the inequality
$p_X^+(f) \ge 2\DD_X^+(f)$. We can assume that $\DD_X^+ (f)>0$.
Fix
 $0<t<2\DD_X^+(f)$.
 By the definition 
 \eqref{e:def-dd-hyd}
 of $\DD_X^+(f)$, there exists $\nu \in \mathcal M^+_X(f)$
  such that $\VD (\nu)>t/2$. Since $\VD (\nu) = (2L_\nu(f))^{-1} h_\nu(f)$
   by Theorem
\ref{thm_main}, we have $h_\nu(f) / L_\nu(f) > t$. It follows that
$h_\nu(f) - t L_\nu(f) >0$; that is, $P_X^+(t)>0$. Therefore we have $p_X^+(f)>t$.
Since $t$ is arbitrary, we obtain $p_X^+(f) \ge 2\DD_X^+(f)$.

\medskip

Let us now prove that $2\DD_X^+(f) \geq p_X^+(f)$. Suppose that $2\DD_X^+(f) < p_X^+(f)$.
Then there exists $t \in (2\DD_X^+(f), p_X^+(f))$ such that $P_X^+(t) > 0$. In particular,
there exists a measure $\nu\in \mathcal{M}^+_X (f)$ with $h_\nu(f) - t L_\nu(f) >0$.
We deduce from Theorem \ref{thm_main}
that
\[
\VD (\nu) = \frac{h_\nu(f)}{2L_\nu(f)} > \frac{t}{2} > \DD_X^+(f).
\]
This contradicts the definition of $\DD_X^+(f)$. Hence, we have $2\DD_X^+(f)\geq p_X^+(f)$.

\medskip
By Lemma \ref{lem_vdbdd}, 
we have $\DD_X^+(f) \le 1$.
Since the function
$t \mapsto P_X^+(t)$ is convex and non-increasing, the equality 
$p_X^+(f) = 
2\DD_X^+(f) \le 2$ implies that the set $\{t\colon  P_X^+(t)=0\}$ is non-empty. The proof is complete.
\end{proof}

\begin{rmk}
One can also define 
\begin{equation*}
	\begin{aligned}
		\DD^+_e (f) & \defeq \sup  \{ \VD (\nu) \colon \nu \in  \mathcal{M}_e^+ (f) \},\\
		P^+_e (t) & \defeq \sup\{h_\nu(f) - t L_{\nu} (f)\colon \nu \in \mathcal{M}_e^+ (f)\}, \quad \mbox{ and }\\
		p^+_e(f) &  \defeq \inf \{t : P^+_e(f) \le 0 \},
	\end{aligned}
\end{equation*}
where we recall that $\mathcal M^+_e (f)$
is  the set of ergodic $f$-invariant measures whose measure-theoretic entropy is strictly larger than $(k-1)\log d$.
Since $\mathcal M_e^+(f)\subseteq \mathcal M^+ (f)$
for every $f$ 
\cite{DeThelin08,Dupont12}, Theorem \ref{thm_main} applies in particular to every $\nu \in \mathcal M^+_e (f)$.
The same proof as Proposition \ref{prop_zero} gives $p_e^+(f) = 
2\DD_e^+(f)$. 
\end{rmk}

\subsection{Proof of Theorem \ref{thm_main_equalities_hyp}}
\label{ss:1.3}

Recall that if $X \subset \bbP^k$ is a 
uniformly expanding 
closed invariant  set for $f$,
the volume dimension $\VD(X)$ is defined as $\VD(X) \defeq \sup_{\nu \in \mathcal{M}_X^+(f)} \VD_{\nu}(X)$;
see Definition \ref{d:vd-hyp}.
\begin{prop} \label{prop_confmeasure2}
Let $X\subseteq \mathbb P^k$
 be a 
 uniformly expanding
 closed invariant  set for $f$
 containing a dense orbit.
We have $\delta_X(f) \ge 2\VD (X)$. In particular, if $f$ is hyperbolic, we have $\delta_J (f)\geq 2\VD (J(f))$.
\end{prop}

\begin{proof}
We can assume that a volume-conformal measure on $X$ exists, otherwise
we have $\delta_X (f)=+\infty$ and the assertion is trivial.
Let $\mu$ be a $t$-volume-conformal measure on $X$, for some
 $t \ge \delta_X(f)$. Since $X$ contains a dense orbit, we have $\text{Supp } \mu = X$.
It suffices to prove
the inequality
 $t  \geq 2\VD_{\nu}(X)$ for any
measure $\nu \in \mathcal{M}_X^+(f)$.
By Definition \ref{d:vd-hyp}, this implies that
 $t \geq 2\VD(X)$
and the conclusion follows by taking the infimum over $t$ as above.

\medskip

Fix $\nu \in \mathcal{M}_X^+(f)$. We can assume that $\VD_\nu (X)>0$,
since otherwise the assertion is trivial. In particular, recalling that all measures in $\mathcal M^+(f)$
are ergodic, we can assume that $\nu$ is non-atomic. Fix  a constant $\gamma>1$.
Since $\VD_\nu (X)=\limsup_{\epsilon \to 0} \VD_\nu^\epsilon (X^\epsilon)$ by Definition \ref{d:vd}
and $\VD_\nu (X)\leq 1$ by Lemma \ref{lem_vdbdd}, we can fix $\epsilon_0=\epsilon_0 (\gamma)$
 such that $\VD_\nu (X)\leq \gamma \VD_\nu^{\epsilon_0} (X^{\epsilon_0})$.
As we can assume that $\epsilon_0(\gamma)\to 0$ as $\gamma\to 1$,
  it is 
enough to prove that $2\VD_\nu^{\epsilon_0}(X^{\epsilon_0})\leq t { \ell(\epsilon_0)}$,
where $\ell(\epsilon)$ is as in \eqref{eq:def-ell}.

\medskip

By Remark \ref{rmk:s4-hyperbolic},
 for every $0<\epsilon\ll \chi_{\min}$
we have $X^\epsilon = X$.
 In particular, 
$X^{\epsilon_0}= X$ is compact.
Fix $\alpha>1$. As in the proof of Proposition \ref{prop_2.1}, since $\mu(X)>0$, for every
 $\eta > 0$ and $0<\kappa<r(\epsilon)$
  there exists an $N^\star$ (depending on $\eta$ and $\kappa$) large enough
  and
   a cover $\{U_i\}_{i\geq 1} \subset 
   \mathcal U (\epsilon, \kappa, N^\star)$
of $X^{\epsilon_0}=X$ satisfying
\begin{equation}\label{eq:sum-mu}
\sum_{i}\mu(U_i)^{\alpha} \leq \eta. 
\end{equation}
As $X^{\epsilon_0}=X$ 
is compact, we can assume that the cover $\{U_i\}$ is finite.
By Lemma \ref{l:mcm-conformal}, 
we have
\begin{equation}\label{eq:sum-vol}
\sum_{i} \volume (U_i)^{t\alpha/2}  \leq 
 \sum_{i} 
 \frac{ C^{t\alpha}\kappa^{tk\alpha} e^{\alpha tN(U_i)k\epsilon(5M+2)}}
{ m_- (\mu, \kappa \, e^{-N(U_i) M \epsilon})^\alpha}
 \mu(U_i)^{\alpha}\\
 \leq \frac{ C^{t\alpha}\kappa^{tk\alpha} e^{\alpha t N^+ k\epsilon(5M+2)}}
{ m_- (\mu, \kappa \, e^{-N^+ M \epsilon})^\alpha}
\sum_{i} 
\mu(U_i)^{\alpha},
\end{equation}
where $m_->0$ is as in Lemma \ref{l:conformal}, $C<\infty$ 
is as in Lemma \ref{l:mcm-conformal},
  and $N^+$ is the maximum of 
  the $N(U_i)$ 
(we use here that the cover $\{U_i\}$ is finite).
We deduce from  \eqref{eq:sum-mu} and \eqref{eq:sum-vol}
that $$\sum_{i} \volume (U_i)^{t\alpha/2} < \frac{ C^{t\alpha}\kappa^{tk\alpha} e^{\alpha t N^+ k\epsilon(5M+2)}}
{ m_- (\mu, \kappa \, e^{-N^+ M \epsilon})^\alpha} \eta < \infty.$$ 
This implies that $\VD^{\epsilon_0, \kappa}_\nu (X^{\epsilon_0})\leq t \alpha/2$ for all $\alpha>1$. 
Taking $\alpha \searrow 1$, we have $\VD^{\epsilon_0, \kappa}_\nu (X^{\epsilon_0})\leq t/2$
for all $0<\kappa<r(\epsilon)$. By Lemma \ref{lem:vdek-vde},
we have $\VD^{\epsilon_0}_\nu (X^{\epsilon_0})\leq \ell(\epsilon_0) t/2$.
The proof is complete.
\end{proof}

The following result implies Theorem \ref{thm_main_equalities_hyp} by taking $X=J(f)$
if $f$ is hyperbolic (since $J(f)$ has dense orbits).
 
\begin{thm}\label{t:hyp-subset}
 Let $f$ be an endomorphism of $\mathbb P^k$ of algebraic degree $d\geq 2$
 and 
 $X\subseteq \bbP^k$ a closed invariant uniformly expanding set containing a dense orbit. Then
 \[
\delta_X(f)
= p^+_{X}(f) = 2\VD (X)
 \]
 and there exists a unique
  invariant probability
 measure $\mu_X$ supported on $X$ 
 and such that $\VD (\mu_X)= \VD (X)$.
 \end{thm}
 
 \begin{proof}
It follows from the general theory
 of thermodynamic formalism for uniformly expanding systems 
 (see for instance \cite{Bowen75} and \cite[Chapters 3 and 6]{PU})
that,
 for every $t\in \mathbb R$,
 there exists a
 unique
 invariant probability 
 measure $\mu_t$  on $X$
 maximizing the pressure function $P_X(t)= \sup_{\nu} \{ h_\nu(f) - t L_\nu(f)\}$, where
 the supremum is taken over all invariant measures $\nu$ on $X$. 
 We used here the fact that, since $X$ is uniformly expanding,
 the function $-t\log |\Jac f|$
 is H\"older continuous on $X$.
 
 Let $\mu_X$ be the invariant measure
 $\mu_{t_0}$ associated to $t_0= p^+_X(f)$.
 As $h_{\mu_X}(f) = p^+_X(f) L_{\mu_X}(f)$,
 it follows from Theorem \ref{thm_main} 
 that we have $p^+_X(f)=2\VD (\mu_X)$. Since  $\VD (\mu_X) \le \VD_{\mu_X}(X)$
 by Definition \ref{d:vd-nu} 
 and $\VD(X) \defeq \sup_{\nu \in \mathcal{M}_X^+(f)} \VD_{\nu}(X)$
 by Definition \ref{d:vd-hyp},
 we have  
 \begin{equation*}\label{eq_VDX}
 \VD (\mu_X)\leq \VD(X).
 \end{equation*}
We deduce from the above and Proposition \ref{prop_confmeasure2}
 that
 \[
 p^+_X(f) = 2\VD (\mu_X) \leq 2\VD(X) \leq \delta_X(f).
 \]
 To complete the proof, we prove the inequality
 $\delta_X(f) \leq p^+_X(f)$ by constructing a $t^\star$-volume conformal measure on $X$ for some $t^\star \leq p^+_X(f)$. 
 Since one can follow Patterson's \cite{Patterson76} and 
 Sullivan's \cite{Sullivan83} constructions of conformal measures, we only sketch the proof and refer to
 those papers
for more details; see also \cite[Sections 12.1 and 12.3]{PU}. 
 
 \medskip
 
Take $x \in X$. For each $m \ge 0$, set
$$E_m \defeq f|_{X}^{-m}(x).$$
Then $E_m$ is finite and $E_{m+1} = f|_{X}^{-1}(E_{m})$. For all $t \ge 0$,
consider the sequence $\{a_m(t)\}_{m \ge 1}$ given by
$$a_m(t) \defeq \log \Big(\sum_{x \in E_m}e^{S_m\phi_t(x)} \Big)$$
where $\phi_t(x) \defeq -t\log|\Jac f(x)|$ 
and $S_m\phi_t(x) \defeq \sum_{j=0}^{m-1} (\phi_t \circ f^j)(x)$. 
Let $c(t)$ be 
defined as
$$c(t) \defeq \limsup_{m \to \infty} \frac{a_m(t)}{m}.$$ 
As 
a consequence of the expansiveness of $f|_{X}$
one can prove that
\begin{equation}\label{e:cP}
c(t) \le P_X(t) \quad  \mbox{ for all } t \ge 0;
\end{equation}
see for instance \cite[Lemmas 12.2.3 and 12.2.4]{PU}. 
Moreover, the function $t \mapsto c(t)$ is continuous.
Setting $$t^\star \defeq \inf \{t \ge 0 :  c(t) \le 0 \},$$
it follows from \eqref{e:cP} that
$t^\star \le p_X^+(f)<\infty$.

\medskip

By \cite[Lemma 12.1.2]{PU},
 there exists a sequence $\{b_m\}_{m\geq 1}$ of positive real numbers such that
  the quantity
\[
M_s \defeq \sum_{m=1}^\infty b_m e^{a_m - ms}
\]
satisfies 
$M_s <\infty$ for $s>t^\star$ and $M_s=+\infty$ for $s\leq t^\star$.
For $s>t^\star$ consider the measure $\nu_s$ defined as
\[\nu_s \defeq \frac{1}{M_s}\sum_{m=1}^\infty \sum_{x\in E_m} b_m e^{a_m - ms}\delta_x.\]
One can check that any 
weak limit of the measures
$\nu_s$ as $s \searrow t^\star$ is
$t^\star$-volume-conformal; see for instance \cite[Section 12.1]{PU}.
The assertion follows.
 \end{proof}


\printbibliography
\end{document}